\documentclass[12pt,a4paper,english]{amsart}
\usepackage{graphicx}
\usepackage[left=20mm, right=20mm, top=25mm, bottom=25mm]{geometry}
\usepackage{color}
\usepackage{multicol}
\usepackage{mathrsfs}
\usepackage{enumitem}
\usepackage{indentfirst, amsfonts, amsmath, amsthm, amscd,amssymb}
\usepackage{epsfig}
\usepackage{hyperref}
\usepackage{amsmath}
\usepackage{amsthm}
\usepackage{amssymb}
\usepackage{amscd}
\usepackage{amsfonts}

\newcommand{\R}{\mathbb{R}} %numeros reais
\newcommand{\N}{\mathbb{N}} %numeros Naturais
\newcommand{\Z}{\mathbb{Z}} %numeros Inteiros

\newtheorem{theorem}{{Theorem}}[section]

\newtheorem{lema}[theorem]{{Lemma}}
\newtheorem{corol}[theorem]{{Corollary}}

\newtheorem{defin}[theorem]{{Definition}}
\newtheorem{prop}[theorem]{{Proposition}}

\newtheorem{maintheorem}{{Main Theorem}}

\newenvironment{sproof}{%
  \proof}{\endproof}

  \newenvironment{remark}{%
  \proof}

\def \A {{\mathcal A}}

\def \N{\mathbb N}

\def \O {{\mathcal O}}

\def \B {{\mathcal B}}
\def \dist {\text {dist}}
\def \Z {{\mathbb Z}}

\def \R {{\mathbb R}}

\title[Dimension of uniform attractors with infinite dimensional symbol space]{Finite fractal dimension of uniform attractors for non-autonomous dynamical systems with infinite dimensional symbol space}

\author{A. N. Carvalho, J. A. Langa and R. O. Moura}
\address[ANC and ROM]{
Instituto de Ci\^{e}ncias Ma\-te\-m\'{a}\-ti\-cas e de Computa\c{c}\~{a}o 
Universidade de S\~{a}o Paulo, Campus de S\~{a}o Carlos, Caixa Postal 668, S\~{a}o Carlos SP, Brazil.}
\email[ANC]{andcarva@icmc.usp.br}
\email[ROM]{rfomoura@usp.br}
\address[JAL]{Depto. Ecuaciones Diferenciales y Anal. Num.\\
Facultad de Matem\'{a}ticas\\
Universidad de Sevilla\\
C/ Tarfia s/n\\
41012-Sevilla (Spain)}
\email[JAL]{langa@us.es}


\thanks{[ANC] Partially supported by FAPESP Grant \# 20/14075-6 and by CNPq Grant \# 308902/2023-8, Brazil}
\thanks{[JAL] Partially supported by Grant \# PID2021-122991NB-C21 Ministerio de Ciencia e Innovaci\'on, Spain}
\thanks{[ROM] Supported  by Grants \#2022/04886-2 and \#2023/11798-5, São Paulo Research Foundation (FAPESP)}

\begin{document}

%\begin{abstract}
%In this paper we consider general abstract results for a class of non-autonomous evolution problems for which the dimension of the symbol space has infinite fractal dimension or is not even compact but the uniform attractor has finite fractal dimension. 		
%\end{abstract}
%
\begin{abstract}
The aim of this paper is to find an upper bound for the box-counting dimension of uniform attractors for non-autonomous dynamical systems. Contrary to the results in literature, we do not ask the symbol space to have finite box-counting dimension. Instead, we ask a condition on the semi-continuity of pullback attractors of the system as time goes to infinity. This semi-continuity can be achieved if we suppose the existence of finite-dimensional exponential uniform attractors for the limit symbols.

After showing these new results, we apply them to study the box-counting dimension of the uniform attractor for a reaction-diffusion equation, and we find a specific forcing term such that the symbol space has infinite box-counting dimension but the uniform attractor has finite box-counting dimension anyway.
\end{abstract}

	\maketitle
	
	\thispagestyle{empty} % <-- you need this for the first page
%-----------------------------------------------------------------------

\section{Introduction}

The study of qualitative properties of non-autonomous dynamical systems in infinite dimension Banach spaces has received a lot of attention in literature (\cite{kpm}, \cite{booknolanovo}, \cite{klo}, \cite{efendiev}, \cite{chep}, \cite{nola}). These systems model several kinds of real-life phenomena in areas like Physics, Chemistry, Biology, Economy and even Social Sciences. 

In particular, the theory of attractors for autonomous dynamical systems and pullback attractors, uniform attractors and skew-product attractors for non-autonomous dynamical systems originated a very fruitful research field (see references above). One aspect that is particularly important is the possibility of describing these attractors as finite dimensional objects estimating their fractal dimension (\cite{embed}). For uniform attractors those estimates are available when the space of symbols has finite fractal dimension which considerably restricts the applicability of the results (\cite{cuietal}).

In this paper we direct our efforts to the cases in which the symbol space has infinite fractal dimension or is not even compact but the uniform attractor has finite fractal dimension. There are, of course, cases for which the dimension of the symbol space works as a lower bound for the fractal dimension of uniform attractor but, as we show in this paper, there are also many situations for which the dimension of the symbol space plays no role in the dimension of the uniform attractor.

 To motivate our studies and to better explain the results in this paper let us consider a particular model - our results, however, will be abstract and general. Consider the following reaction-diffusion equation:
\begin{equation} \label{intazim}
\begin{split}
&v_t  = \Delta v + f( v )+g(t),\quad x\in\mathcal{O},\ t>\tau,\\
&v (x,\tau)= v _{\tau}(x),\quad x\in\mathcal{O},\\
&v (x,t)=0,\quad x\in\partial\mathcal{O},\ t\geqslant 0,
\end{split} 
\end{equation}

\noindent where ${\mathcal{O}}\subset\mathbb{R}^{N}$, ($N\in\mathbb{N}$) is a bounded domain with smooth boundary $\partial\Omega$. 

Let $X = (L^2(\O), \| \cdot \|)$ and $\Xi:={\mathcal{C}}(\mathbb{R},L^2(\O)),$ the space of all continuous functions $\xi:\mathbb{R} \rightarrow L^2(\O),$ endowed with the Fréchet metric $d_{\Xi}$:
$$
d_{\Xi}(\xi_{1},\xi_{2}):=\sum_{n=0}^{\infty}\frac{1}{2^{n}}\frac{d^{(n)}(\xi_{1},\xi_{2})}{1+d^{(n)}(\xi_{1},\xi_{2})},\qquad\xi_{1},\xi_{2}\in\Xi, 
$$
where
$$
d^{(n)}(\xi_{1},\xi_{2}):=\max_{s\in[-n,n]} \|\xi_1(s) - \xi_2(s)\|,\qquad n\in\mathbb{N}. 
$$

We define the hull $\mathcal{H}_\Xi(g)$ of the non-autonomous function $g\in\Xi$ as:
$$
\mathcal{H}_\Xi(g):={\overline{{\{\theta_r g:r\in\mathbb{R}\}}}}, 
$$
where $\theta_r g = g(\cdot + r)$ and the closure is taken under the metric $d_{\Xi}$ of $\Xi$.

Under suitable assumptions (that will be stated later) on the  functions $f$ and $g$, for any $\sigma \in \Sigma = \mathcal{H}_\Xi(g)$, $\tau \in \R$, the initial boundary value problem
\begin{equation} \label{intmarjax}
\begin{split}
&v_t  = \Delta v + f( v )+\sigma(t),\quad x\in\mathcal{O},\ t>\tau,\\
&v (x,\tau)= v _{\tau}(x),\quad x\in\mathcal{O},\\
&v (x,t)=0,\quad x\in\partial\mathcal{O},\ t\geqslant 0,
\end{split} 
\end{equation}
possesses a unique solution $u(\cdot,\tau,u_0)$ defined for all $t\geqslant \tau$ and \eqref{intmarjax} gives rise to an evolution process $\{U_\sigma(t,\tau): t\geq \tau\}$ where $U_\sigma( \cdot,\tau) v_\tau =u(\cdot,\tau,v_\tau)$. Therefore, for any $\tau \in \R$, $t\geq \tau$ we can consider the operator $U_\sigma(t,\tau) \in \mathcal{C}(X)$, and we call $\{U_\sigma(t,\tau): t\geq \tau\}_{\sigma\in \Sigma}$ the system of processes associated to \eqref{intazim}. Studying this system of processes, we can capture all the asymptotic behavior of the solutions of \eqref{intazim}.

Assuming some additional dissipative conditions on $f$ and $g$ (that will be stated later), we can obtain the pullback attractor of equation \eqref{intazim} which is a family of compact sets $\{A(t):t\in \R\}$, $A(t) \subset X$ for all $t\in \R$, such that $\bigcup_{t\in \R} A(t)$ is bounded, $U_g(t,s)A(s) = A(t)$ for all $t\geq s$, and for every bounded set $B\subset X$ and $t \in \R$, we have $\dist_H(U_g(t,s)B, A(t)) \overset{s\to -\infty}{\longrightarrow} 0$, where $\dist_H(A,B) = \sup_{x\in A} \dist(x,B)$ is the Hausdorff semi-distance. The pullback attractor at time $t$ reflects the behavior at time $t$ of solutions that started a long time ago.

Moreover, we can define the uniform attractor as the compact set $\mathcal{A}_\Sigma \subset X$ that is the minimal closed set with the property that for every bounded subset $B\subset X$, we have
$$
\lim_{t\to \infty} \sup_{\sigma\in \Sigma} \dist_H(U_\sigma(t,0)B, \mathcal{A}_\Sigma) = 0.
$$

The uniform attractor has the property of representing the asymptotic forward behavior of equation \eqref{intazim} uniformly with respect to all symbols $\sigma\in \Sigma$ and in particular  represents the asymptotic forward behavior with respect to the elapsed time uniformly for all initial time $\tau \in \R$ .

When the nonlinear functions $f$ and $g$ do not depend on time $t$, we have $\mathcal{H}_\Xi(g) = \{g\}$, and the process $\{U_g(t, s): t\geq s\}$ satisfies $U_g(t+s,s) = U_g(t,0)$ for all $t\geq 0$, $s\in \R$. We then define $T(t) = U_g(t,0)$ and we get a semigroup $\{T(t): t\geq 0\}$ in $X$. In this case, we can define the global attractor of $\{T(t):t\geq 0\}$ as the non-empty, compact set $\A \subset X$ such that $T(t)\A = \A$ for all $t\geq 0$ and for any bounded set $B\subset X$, it holds $\dist_H(T(t)B,\A) \overset{t\to \infty}{\longrightarrow} 0$. In this simplified case, we have $A(t) = \mathcal{A}$ for all $t\in \R$, where $\{A(t):t\in \R\}$ is the pullback attractor of $\{U_g(t,s):t\geq s\}$.

Our main interest lies in the study of the box-counting (fractal) dimension of the uniform attractor $\mathcal{A}_\Sigma$. If $X$ is a metric space, for a compact set $K\subset X$, the box-counting dimension of $K$ is defined as follows:
$$
d_B(K) = \limsup_{r \to 0} \frac{\log N_X(K,r)}{- \log(r)}.
$$
where $N_X(K,r)$ is the minimal number of open balls in $X$ of radius $r > 0$ necessary to cover $K$.

One of the most important applications of the box-counting dimension is the following Theorem: 

\begin{theorem}
Let $X$ be a Banach space, $K\subset X$ be a compact set, $k$ be a positive integer such that $k>2 d_B(K)$, and $\theta \in \mathbb{R}$ such that:
\[
0<\theta< \frac{k-2d_B(K)}{k(1+\alpha d_B(K))},
\]
where $\alpha=1/2$ if $X$ is a Hilbert space and $\alpha=1$ if $X$ is a Banach space. Then there exists a prevalent set of bounded linear operators $L\in \mathcal{L}(X,\mathbb{R}^k)$ that are injective between $K$ and their image, and such that there exists $c_L>0$ such that 
\[
\|x-y\| \leq c_L |Lx-Ly|^\theta, \quad \forall \, x,y\in K.
\]

In other words, the inverse of $L$ defined on $LK$ is Hölder continuous.
\end{theorem}

A first version of this theorem was shown in \cite{soso}, and the result was improved in \cite{sosos,sososs}. The theorem shows that a uniform attractor with finite box-counting dimension could be injectively projected into a finite-dimensional space, and the inverse is Hölder continuous, which implies a limitation on the complexity of the dynamical objects contained in this attractor.

Therefore, it is important to find conditions under which attractors have finite box-counting dimension.

The authors in \cite{bispace} show that a broad class of semigroups has a global attractor with finite box-counting dimension. The main hypothesis — which we call \textit{smoothing} — is that for some $t_0\geq 0$, the semigroup $S(t_0)$ can be decomposed as the sum $S(t_0)=L+C$ where $L \in \mathcal{L}(X)$ satisfies $\|L\|<1$ and the operator $C$ compactifies the phase plane, more precisely, there exists a space $Y$ such that $X$ is compactly embedded in $Y$, and the following holds:
$$
\|Cx-Cy\|_X \leq \|x-y\|_Y, \quad x,y \in \mathcal{A}.
$$

A study on the box-counting dimension of pullback attractors is done in \cite{estefanizona}, similarly to the study of the box-counting dimension of global attractors in \cite{bispace}. The study of the box-counting dimension of uniform attractors of systems of evolution processes is being carried out recently, with one of the latest contributions in the articles \cite{cuietal, efendiev}, where the authors proved that the uniform attractor has finite box-counting dimension under the assumption that the symbol space $\Sigma$ has finite box-counting dimension in $\Xi$ (for example, this happens when $g$ is a quasi-periodic function — see \cite{chep}). This assumption is quite restrictive, and we wish to refine these results to avoid asking $\Sigma$ to be finite dimensional. Indeed, we will give a non-trivial example — actually, a class of examples — where the box-counting dimension of the symbol space $\Sigma$ is infinite, but the uniform attractor $\mathcal{A}_\Sigma$ has finite box-counting dimension.

Our strategy will be as follows. We will consider the case where $\Sigma = \mathcal{H}_\Xi(g)$ can be decomposed as:
\begin{equation}\label{intgorux}
\mathcal{H}_\Xi(g) = \overline{\{\theta_t g :t\in \R\}} \subset \Sigma_- \cup \{\theta_t g : t\in \R\} \cup \Sigma_+.
\end{equation}
where $\Sigma_-$, $\Sigma_+\subset \Xi$. This happens, for example, when $g$ approaches functions $g_-$ as $t\to -\infty$ and $g_+$ as $t\to \infty$, so we are able to decompose $\Sigma$ as in \eqref{intgorux}, with $\Sigma_- = \mathcal{H}_\Xi(g_-)$ and $\Sigma_+ = \mathcal{H}_\Xi(g_+)$. More precisely, if we define
$$
\alpha(g) = \{\sigma: \text{ there exists a sequence } t_n \overset{n\to \infty}{\longrightarrow} - \infty \text{ such that } \theta_{t_n}g \overset{n\to \infty}{\longrightarrow} \sigma \}
$$
$$
\omega(g)= \{\sigma: \text{ there exists a sequence } t_n \overset{n\to \infty}{\longrightarrow} \infty \text{ such that } \theta_{t_n}g \overset{n\to \infty}{\longrightarrow} \sigma \}
$$
then it is easy to show that \eqref{intgorux} happens if $\alpha(g)\subset \Sigma_-  $ and $\omega(g) \subset \Sigma_+  $.

Since we can prove that (\cite{booknolanovo}) 
$$
\mathcal{A}_\Sigma = \bigcup_{\sigma\in \Sigma} A_\sigma(0)
$$
where $A_\sigma(0)$ is the pullback attractor of $\{U_\sigma(t,s):t\geq s\}$ at time $0$, we can use \eqref{intgorux} to obtain:
\begin{equation}
\mathcal{A}_\Sigma \subset \mathcal{A}_{\Sigma_-} \cup \mathcal{A}_{\Sigma_+} \cup \bigcup_{t\in \R} A_{\theta_t g}(0) = \mathcal{A}_{\Sigma_-} \cup \mathcal{A}_{\Sigma_+} \cup \bigcup_{t\in \R} A_{g}(t) .
\end{equation}

When $\Sigma_-$ and $\Sigma_+$ have finite box-counting dimension (for example, when they are the hull of quasi-periodic functions $g_-$ and $g_+$), we can use the existing theorems in \cite{cuietal, efendiev} to show that $\mathcal{A}_{\Sigma_-}$ and $\mathcal{A}_{\Sigma_+}$ have finite box-counting dimension as well, and the only set that we need to worry about is 
\begin{equation}\label{intzoooox}
\bigcup_{t\in \R} A_{g}(t).
\end{equation}

In \cite{estefanizona}, it is proven that, under the smoothing conditions, the box-counting dimension of the sets $A_g(t)$ can be bounded uniformly in $t\in \R$. However, as presented in \cite[Chapter 6]{galox}, the union \eqref{intzoooox} can have infinite box-counting dimension even when the box-counting dimension of the sets $A_g(t)$ are uniformly controlled. To bound the box-counting dimension of this union, we ask that the pullback attractors accumulate into sets $\mathcal{A}_+$ and $\mathcal{A}_-$ at an exponential or polynomial rate, and that the evolution process $\{U_\sigma(t,s): t\geq s\}$ is Hölder continuous in time variable. This result is shown below:

\begin{maintheorem}\label{intdenras}
    Let $\{U_\sigma(t,s)\}_{\sigma \in \Sigma}$ be a system of processes in $X$ and for $\psi \in \Sigma$, suppose that $\{U_\psi(t,s)\}$ has a pullback attractor $\{A(t): t\in \R\}$. Suppose that $\mathcal{B}$ is a closed, bounded uniformly absorbing set of the system such that the following properties hold:

\begin{enumerate}[label=$(H_{\arabic*})$]
    \item \label{intaux} There is an auxiliary Banach space $Y$ compactly embedded in $X$.
     \item \label{intsmoothing}  $\{U_\sigma(t,s)\}_{\sigma \in \Sigma}$ is uniformly smoothing on the set $\mathcal{B}$, that is, for any $t>0$, there exists $\kappa(t)$ such that
     $$
    \underset{\sigma\in \Sigma}{\sup} \|U_\sigma(t,0)x - U_\sigma(t,0)y\|_Y \leq \kappa(t) \|x-y\|_X, \quad \forall \, x,y \in \mathcal{B}.
     $$

     \item \label{intholder} For $s\in \R$, $x$, $y\in \mathcal{B}$, we have:
     $$
     \| U_\sigma(t+s,s)x - U_\sigma(\Tilde{t}+s,s)y \| \leq C(r) (\|x-y\|^\gamma + |t-\Tilde{t}|^\theta),
    $$ for $t$, $\Tilde{t} \in [0,r]$ where $C(r)$ is uniform for $s \in \R$, $\sigma \in \Sigma$.
\end{enumerate}

 \begin{enumerate}[label=$(A_{\arabic*})$]
     \item \label{intentering} There exist asymptotic sets $\mathcal{A}_-$ and $\mathcal{A}_+$ such that $d_B(\mathcal{A}_-) = d_- < \infty$ and $d_B(\mathcal{A}_+) = d_+ <\infty$ and a strictly decreasing function $f: \R^+ \to \R^+$, $f(t) \overset{t\to \infty}{\longrightarrow} 0$ and $\Tilde{T}\geq 0$ such that
     $$\dist_H(A(t), \mathcal{A}_+) \leq f(t), \quad \forall \, t \geq \Tilde{T}$$ and $$\dist_H(A(t), \mathcal{A}_-) \leq f(|t|), \quad \forall \, t \leq -\Tilde{T}.$$
 \end{enumerate}

    Moreover, let $\tau$ be the time when $\mathcal{B}$ absorbs itself, that is:
    \begin{equation}\label{intfatal}
    \bigcup_{\sigma\in \Sigma} U_\sigma(t,0)\mathcal{B} \subset \B , \quad \forall \, t\geq \tau,  
    \end{equation}
and let $N_{\frac{1}{4\kappa(\tau)}}$ denote the number of balls of radius $\frac{1}{4\kappa(\tau)}$ in $X$ necessary to cover the unit ball $B_Y(0,1)$ in $Y$.

    Then the following happens:

     \begin{enumerate}[label=$(T_{\arabic*})$]
         \item \label{intT1} If $f(t) = K / t^r$, for some $K \geq 1$ and $r\in (0,\infty)$, we have:
         \begin{equation*}
            d_B\left ( \bigcup_{t\in \R} A(t) \right) \leq \max \left (d_+, \, d_-, \, \frac{1}{r} + \frac{1}{\theta} + \frac{1}{\gamma}\cdot \log_2 N_{\frac{1}{4\kappa(\tau)}} \right )
         \end{equation*}
         \item \label{intT2} If $f(t) = K e^{- \xi t}$, for some $K\geq 1$ and $\xi > 0$, 
         \begin{equation*}
             d_B\left ( \bigcup_{t\in \R} A(t) \right) \leq \max \left (d_+, \, d_-, \, \frac{1}{\theta} + \frac{1}{\gamma}\cdot \log_2 N_{\frac{1}{4\kappa(\tau)}} \right )
         \end{equation*}
     \end{enumerate}

\end{maintheorem}

Theorem \ref{intdenras} handles the box-counting dimension of \eqref{intzoooox} when the sets $A_g(t)$ converge to asymp\-totic sets $\mathcal{A}_+$ and $\mathcal{A}_-$ at a polynomial or exponential rate $f(t)$. Then we need to find conditions to ensure that this convergence occur. In order to obtain this, we can treat the pullback attractors $A_g(t)$ as perturbations of the limit attractor $\mathcal{A}_+ = \mathcal{A}_{\Sigma_+}$, using $\Sigma_+ = \mathcal{H}_\Xi(g_+)$, where $d_\Xi(\theta_t g, \theta_t g_+) \overset{t\to \infty}{\longrightarrow} 0$. This allows us to use a similar proof as
\cite[Theorem 8.2.1]{babin} (which is a theorem about upper semi-continuity of perturbed attractors). 

With this reasoning, we obtain Proposition \ref{barol}, that states roughly that if there exists $\Sigma_-$ and $\Sigma_+$ that are compact invariant subsets of $\Sigma$ such that, for constants $C\geq 1$, $\eta_1$, $\eta_2 > 0$,
$$
\dist(\theta_t g, \Sigma_-) \leq C e^{\eta_1 t},\quad t\in \R,
$$
$$
\dist(\theta_t g, \Sigma_+) \leq C e^{-\eta_2 t},\quad t\in \R,
$$
and, moreover, there exist \textit{exponential uniform attractors} $\mathcal{M}_{\Sigma_-}$ and $\mathcal{M}_{\Sigma_+}$, then we have the semi-continuity property \ref{intentering} with $\mathcal{A}_- = \mathcal{M}_{\Sigma_-}$ and $\mathcal{A}_+ = \mathcal{M}_{\Sigma_+}$, $f(t) = K e^{- \xi t}$, for some $K\geq 1$ and $\xi > 0$, and Theorem \ref{intdenras} can be applied whenever $\mathcal{M}_{\Sigma_-}$ and $\mathcal{M}_{\Sigma_+}$ have finite box-counting dimension.

Then we have the problem of finding exponential uniform attractors for the two processes $\{U_\sigma\}_{\sigma\in \Sigma_-}$ and $\{U_\sigma\}_{\sigma\in \Sigma_+}$. This can be achieved supposing that $\Sigma_-$ and $\Sigma_+$ have finite box-counting dimension, which is the content of Theorem \ref{ralec}. This theorem is based on the work of Efendiev, Miranville and Zelik in \cite{efendiev}, where the authors estimate the Kolmogorov entropy of uniform attractors for a reaction-diffusion equation. Our result \ref{ralec} is abstract and focuses on the case of finite dimensional symbol spaces.

With all these preliminary results, we are able to prove our second main Theorem:

\begin{maintheorem}\label{inthuuuuglou}
Let $\{U_\sigma(t,s):t\geq s\}_{\sigma\in \Sigma}$ be a system of evolution processes in the metric space $X$, with $\Sigma = \mathcal{H}_\Xi(g)$ and $\{\theta_t: t\in \R\}$ a group of operators in $\mathcal{C}(\Xi)$ such that $(t,\sigma) \mapsto \theta_t \sigma$ is continuous. Suppose that the system satisfies \ref{intaux} - \ref{intholder}, and also
\begin{enumerate}[label=$(H_{\arabic*})$, start=4]
    \item \label{intlipdriving} The driving semigroup is Lipschitz for every $t\geq 0$. More precisely, there exist constants $P \geq 1$ and $\zeta > 0$ such that for all $t\geq 0$,
    $$
    \dist_\Xi(\theta_t \xi_1, \theta_t \xi_2) \leq P e^{\zeta t} \dist_\Xi(\xi_1, \xi_2), \quad \forall \xi_1, \, \xi_2\in \Xi.
    $$
    \item \label{intlipschitz} $\{U_{\sigma}(t, s)\}_{\sigma \in \Sigma}$ is $(\Sigma,X)$-Lipschitz on the absorbing set $\mathcal{B}$, satisfying
	\begin{equation*}\label{intlipschitzsymbol}
		\| U_{\sigma_{1}}(t,0) x - U_{\sigma_{2}}(t,0) x \|_X \leqslant L(t) d_\Xi (\sigma_{1} , \sigma_{2} ), \qquad \forall t \geq 1, \, \sigma_{1}, \sigma_{2} \in \Sigma, \, x \in \mathcal{B}, 
	\end{equation*}
	where $1\leq L(t) \leq c_1 e^{\beta t}$ for some positive constants $ c_1 ,\beta > 0$ and  for $t \geq 1$.
\end{enumerate}

Suppose also that there exist finite dimensional compact invariant sets $\Sigma_-$ and $\Sigma_+ \subset \Xi$ such that $\alpha(g) \subset \Sigma_-$, $\omega(g) \subset \Sigma_+$ and there exist constants $\eta_1 >0$, $\eta_2 >0$, $C\geq 1$ such that:
\begin{equation}\label{intblux}
\dist(\theta_t g, \Sigma_-) \leq C e^{\eta_1 t}
\end{equation}
\begin{equation}\label{intravael}
\dist(\theta_t g, \Sigma_+) \leq C e^{-\eta_2 t}
\end{equation}

Then the system $\{U_\sigma(t,s):t\geq s\}_{\sigma\in \Sigma}$ has a finite dimensional uniform attractor, satisfying the estimate:
\begin{equation}\label{intatroz}
  d_B(\mathcal{A}_\Sigma) \leq \frac{1}{\theta} + \frac{1}{\gamma} \cdot \log_2 N_{\frac{1}{4\kappa(\tau)}} + \left ( \frac{(\beta + \zeta) \tau}{\log 2} +1 \right ) \cdot \max(d_B(\Sigma_-), d_B(\Sigma_+)), 
\end{equation}
where $N_{\frac{1}{4\kappa(\tau)}}$ denotes the number of balls of radius $\frac{1}{4\kappa(\tau)}$ in $X$ necessary to cover the unit ball $B_Y(0,1)$ in $Y$.
\end{maintheorem}

\begin{sproof}
As mentioned before, since $\alpha(g) \subset \Sigma_-$, $\omega(g) \subset \Sigma_+$, it can be proven that
\begin{equation*}
\mathcal{A}_\Sigma \subset \mathcal{A}_{\Sigma_-} \cup \mathcal{A}_{\Sigma_+} \cup \bigcup_{t\in \R} A_{\theta_t g}(0) = \mathcal{A}_{\Sigma_-} \cup \mathcal{A}_{\Sigma_+} \cup \bigcup_{t\in \R} A_{g}(t)  
\end{equation*}
Since $\Sigma_-$ and $\Sigma_+$ have finite box-counting dimension, it follows from the results in \cite{cuietal} that $\mathcal{A}_{\Sigma_-}$ and $\mathcal{A}_{\Sigma_+}$ also have finite box-counting dimension. 

Theorem \ref{ralec} gives the existence of exponential uniform attractors $\mathcal{M}_{\Sigma_-}$ and $\mathcal{M}_{\Sigma_+}$, then Proposition \ref{barol} holds. This proposition implies condition \ref{intentering} with an exponential function $f$. Then we can apply Theorem \ref{intazim} to show that
$$
\bigcup_{t\in \R} A_g(t)
$$
has finite box-counting dimension, from which follows that the whole attractor $\mathcal{A}_\Sigma$ has finite box-counting dimension.

\end{sproof}

\begin{remark}
    Notice that if $t\mapsto \theta_t g$ is Lipschitz (or Hölder) continuous (which usually follows when $g$ is Lipschitz), then we can use Theorem 4.8 in \cite{cuietal} to show that $\Sigma$ has finite box-counting dimension, from which follows by Theorem 3.1 of this same article that $\mathcal{A}_\Sigma$ has finite box-counting dimension. The advantage in the theorem presented here is that $t\mapsto \theta_t g$ need not to be Lipschitz or Hölder, so that we do not assume anything on the box-counting dimension of $\Sigma$. The Hölder hypothesis is instead put on the evolution process \ref{intholder}. It is a clear improvement since we can show the Hölder continuity for the evolution processes without asking Hölder continuity for the application $t\mapsto \theta_t g$, and therefore we do not assume anything on the box-counting dimension of $\Sigma$ (only on the dimension of the asymptotic symbols $\Sigma_-$ and $\Sigma_+$).
\end{remark}
\bigskip

The main contribution of the main theorems \ref{intdenras} and \ref{inthuuuuglou} is the uniform estimate of the box-counting dimension of the complete pullback attractor for a given evolution process $\{U(t,s):t\geq s\}$, that is,
$$
\mathcal{A}_p = \bigcup_{t\in \R} A(t)
$$
is finite dimension, under the stated hypotheses. This implies the existence of a projection from $\mathcal{A}_p$ into a finite dimensional Hilbert space that is a bijection onto its image, and has a Hölder continuous inverse (see \cite{embed}). It was already known (see \cite{jung}) that if the standard smoothing hypothesis \ref{intsmoothing} is satisfied and the evolution process is Hölder continuous in time, then we have
$$
d_B \left (\bigcup_{t\in [-n, n]} A(t) \right ) < \infty
$$
for every $n\in \N$, and $\mathcal{A}_p$ has a bijection onto a subset of a finite dimensional Hilbert space. However, without estimating the box-counting dimension of the whole set $\mathcal{A}_p$, we could not guarantee that the inverse of this bijection was Hölder continuous.

Next corollary allows us to estimate the box-counting dimension of the sets
$$
\bigcup_{t\leq n} A(t), \quad n\in \N
$$
under weaker assumptions than in Main Theorem \ref{inthuuuuglou}.

\begin{corol}
    Let $\{U_\sigma(t,s):t\geq s\}_{\sigma\in \Sigma}$ be a system of evolution processes in the metric space $X$, with $\Sigma = \mathcal{H}_\Xi(g)$ and $\{\theta_t: t\in \R\}$ a group of operators in $\mathcal{C}(\Xi)$ such that $(t,\sigma) \mapsto \theta_t \sigma$ is continuous. Suppose that the system satisfies \ref{intaux} - \ref{intlipschitz} and that there exists a finite dimensional compact invariant set $\Sigma_- \subset \Xi$ such that $\alpha(g) \subset \Sigma_-$ and constants $\eta_1 >0$, $C\geq 1$ such that:
\begin{equation}
\dist(\theta_t g, \Sigma_-) \leq C e^{\eta_1 t}
\end{equation}

Then, for any $T\in \R$, we have:
$$
d_B\left ( \bigcup_{t\leq T} A_g(t)\right ) < \infty.
$$
\end{corol}

\bigskip

Now we come back to problem \eqref{intazim}. Consider any function $f \in \mathcal{C}(\R, \R)$ and $g\in \mathcal{C}(\R, X)$ such that \eqref{intazim} has a unique solution $v$ for every initial values $\tau \in \R$, $v_\tau \in X$. Suppose further that $f$ and $g$ are regular enough so that we have a system of processes $\{U_\sigma(t,\tau): t\geq \tau\}_{\sigma\in \Sigma}$, where $\Sigma = \mathcal{H}_\Xi (g)$, $U_\sigma(t,\tau)v_\tau$ is the solution of the problem \eqref{intmarjax}.

With dissipativity assumptions on $f$ and $g$, we can guarantee that this system of processes has an uniform attractor $\A_\Sigma \subset X$ that is bounded in $L^\infty(\O)$ and in $H^1(\O)$.

We suppose that there exist \textit{quasi-periodic functions} $g_\pm \in \Xi$ such that:
\begin{equation}\label{intguarax}
\|g(t)-g_-(t)\| \leq Q_1 e^{\eta_1 t}, \quad \forall t\in \R.   
\end{equation}
\begin{equation}\label{intmutambo}
\|g(t)-g_+(t)\| \leq Q_2 e^{-\eta_2 t}, \quad \forall t\in \R,    
\end{equation}

Let $\Sigma_- = \mathcal{H}_\Xi(g_-)$, $\Sigma_+ = \mathcal{H}_\Xi(g_+)$, $X:=(L^{2}({\mathcal{O}}),\|\cdot\|)$, $Y:=H_{0}^{1}({\mathcal{O}})$. Using theorem \ref{ralec} we can prove that there exist finite-dimensional exponential uniform attractors $\mathcal{M}_{\Sigma_-}$ and $\mathcal{M}_{\Sigma_+}$ for the processes $\{U_\sigma(t,s):t\geq s\}_{\sigma\in \Sigma_-}$ and $\{U_\sigma(t,s):t\geq s\}_{\sigma\in \Sigma_+}$. Moreover, we can use Theorem \eqref{inthuuuuglou} to estimate the box-counting dimension of $\mathcal{A}_\Sigma$ as a depending function on the box-counting dimensions of $\Sigma_-$ and $\Sigma_+$, and this shows how our new results can be applied to a particular problem.

\bigskip

In Section \ref{prel} we will present the definitions and preliminary results necessary to study non-autonomous dynamical systems. Then, in Section \ref{exp}, we will show the necessary result about existence of an exponential uniform attractor $\mathcal{A}_\Sigma$ for the process $\{U_\sigma(t,s):t\geq s\}_{\sigma\in \Sigma}$ when $\Sigma$ is finite dimensional. As we said before, this is an abstract version of the results obtained in \cite{efendiev}. In Section \ref{main}, we will present the original results that culminate on Theorem \ref{inthuuuuglou}, showing how to estimate the box-counting dimension of $\mathcal{A}_\Sigma$ without assuming that $\Sigma$ has a finite box-counting dimension. Finally, in Section \ref{app}, we apply the previous results to study the reaction-diffusion equation \eqref{intazim}.

\bigskip

In particular, we will study a particular case of forcing $g:\R \to X$ in \eqref{intazim} such that $\Sigma = \mathcal{H}_\Xi(g)$ has infinite box-counting dimension. The function $g$ will be defined the following way. Let $A \subset X$ be any compact set with infinite box-counting dimension in $X$ — for example, the unit ball of $H^1_0(\mathcal{O})$ considered as a subspace of $L^2(\mathcal{O})$ is compact and has infinite topological dimension, hence infinite box-counting dimension. It is possible to construct a continuous surjective function from the Cantor set $\mathcal{C} \subset [0,1]$ onto $A$, that is $\alpha:\mathcal{C} \to X$ such that $\alpha(\mathcal{C}) = A$ (see \cite{immer}). We extend this function continuously to $g:\mathbb{R} \to X$ satisfying $g(t) = 0$ for $t\geq T$ or $t\leq -T$, where $T >1$ is a constant. In this case, the image of $g$ will be infinite dimensional, and the box-counting dimension of $\mathcal{H}_\Xi(g)$ will also be infinite. But defining $\Sigma_- = \Sigma_+ = \{\textbf{0}\}$, where $\textbf{0}:\R \to X$ is the zero function, and asking additional hypothesis on the function $f$, we can apply Theorem \ref{inthuuuuglou} and show that equation \eqref{intazim} has an uniform attractor with finite box-counting dimension. This highlights the advantages of our new theorems in relation to the existent results in the literature, which require finite-dimensionality of $\Sigma$.

% ####################################################################################################################################################################################################################################################################################################################################################################################################################################################################################################################################################################################################################################################################################################################################################################################################################################

\section{Definitions and Preliminary Results}\label{prel}

Let $\left(X,\|\cdot\|_{X}\right)$ be a Banach space and let $\{U(t,s):t,s\in\mathbb{R},\, t\geq s \}$ be a two-parameter set of continuous mappings from $X$ into itself. In the following we will give the basic definitions and results related to non-autonomous dynamical systems, including the notions of pullback and uniform attractors.

\begin{defin}
We say that $\{U(t,s): t\geq s\}$ is an \textbf{evolution process} in $X$ if it satisfies:
\begin{itemize}
    \item $U(s,s)=Id_X$, for all $s \in \R$;
    \item $U(t, s)U(s,\tau) = U(t,\tau), \text{ for all } t \geq s \geq \tau$;
    \item For each $x\in X$ the mapping $\{(t,s)\in \R^2: t\geq s\}\ni (t,s)\mapsto U(t,s)x\in X$ is continuous.
\end{itemize}

A particular case of evolution process is when $U(t+\tau,s+\tau) = U(t,s)$ for all $t\geq s$, $\tau \in \R$. In this case, the family $\{T(t):t\geq 0\}$ defined as $T(t) = U(t,0)$, $t\geq 0$ is called a \textbf{semigroup}. 

Equivalently, a family of mappings $\{T(t):t \in \R\}$, $T(t): X\to X$ is called a \textbf{semigroup} if it satisfies the following properties:
\begin{itemize}
    \item $T(0)=Id_X$;
    \item $T(t)T(s) = T(t+s), \text{ for all } t,  s\geq 0$;
    \item For each $x\in X$, the mapping $[0,\infty)\ni t\mapsto T(t)x\in X$ is continuous.
    
\end{itemize}

\end{defin}

The intuition about an evolution process is that we \textit{evolve} the phase space $X$ as follows: for each initial position $x\in X$ and initial time $s\in \R$, $U(t,s)x \in X$ denotes the (final) state at a later time $t$. If the evolution does not depend on the initial time $s\in \R$ (for example, if the vector field guiding the solution does not change with time), we can use a semigroup to model the dynamical system. In this case, for each initial position $x \in X$, and initial time $s\in \R$, we denote the final state (at a later time $t$) as $T(t-s)x$.

For an evolution process, we can define its \textit{pullback attractor}. To that end, we first recall the definition of Hausdorff semi-distance between sets $A\subset X$ and $B\subset X$ as $\dist_H(A,B) := \sup_{x\in A} \dist(x,B)$.

\begin{defin}
Let $\{U(t,s):t\geq s\}$ be an evolution process. The \textbf{pullback attractor} of this process (if it exists) is the family $\{A(t):t\in \R\}$ of compact subsets of $X$ such that:
\begin{itemize}
    \item $\bigcup_{t\in \R} A(t)$ is bounded;
    \item The family $\{A(t):t\in \R\}$ is \textbf{invariant} under the action of the evolution process $\{U(t,s):t\geq s\}$, that is, $U_g(t,s)A(s) = A(t)$ for all $t\geq s$; 
    \item For every $t\in \R$, the set $A(t)$ \textbf{pullback attracts} any bounded subset of $X$ at time $t$, that is, for every bounded set $B\subset X$ and $t \in \R$, we have $\dist_H(U(t,s)B, A(t)) \overset{s\to -\infty}{\longrightarrow} 0$. 
\end{itemize}
\end{defin}

The pullback attractor at time $t$ reflects the asymptotic behavior of solutions that began further and further in the past evolved to the present time $t$.

\begin{defin}
Let $\{T(t):t\geq 0\}$ be a semigroup. The \textbf{global attractor} of this semigroup (if it exists) is the non-empty, compact set $\mathcal{A}$, such that:
\begin{itemize}
    \item The set $\A$ is invariant by the semigroup, that is, $T(t)\A = \A$ for all $t\geq 0$; 
    \item For every bounded set $B\subset X$ we have $\dist_H(T(t)B, \A) \overset{t\to \infty}{\longrightarrow} 0$. 
\end{itemize}
\end{defin}

Now let us define the object that will help us to study non-autonomous dynamical systems like the non-autonomous partial differential equation \eqref{intmarjax}.

Let $\left(\Xi,d_{\Xi}\right)$ be a complete metric space and let $\{\theta_t:t\in \R\}$ be a group of continuous operators acting on  $\Xi$, that is, $\theta_0\sigma=\sigma$ and $\theta_{t}(\theta_{s}\sigma)=\theta_{t+s}\sigma$ for all $\sigma\in\Xi$,  $t$, $s\in\mathbb{R}$, and for each $s\in\mathbb{R}$, $\theta_s:\Xi \to \Xi$ is a continuous mapping. Let $\Sigma\subset\Xi$ be a compact subset of $\Xi$ which is invariant under the action of $\{\theta_{t}: t\in \R\}$, that is, $\theta_{t}\Sigma=\Sigma$ for all $t\in\mathbb{R}$. We recall the definitions of the alpha and omega limits of the point $g\in \Xi$:
$$
\alpha(g) = \left \{y = \lim_{n\to \infty} \theta_{t_n} g \text{ where } t_n \to -\infty \text{ and the limit $y$ exists}\right \}
$$
$$
\omega(g) = \left \{y = \lim_{n\to \infty} \theta_{t_n} g \text{ where } t_n \to \infty \text{ and the limit $y$ exists}\right \} .
$$

Let us consider now a collection of evolution processes $\left\{U_\sigma(t,s):t\geq s\right\}_{\sigma\in\Sigma}$, where each $\{U_\sigma(t,s):t\geq s\}$ is an evolution process in $X$

\begin{defin}
The collection $\left\{U_\sigma(t,s)\right\}_{\sigma\in\Sigma}$ is called a \textbf{system of evolution processes} if the translation property
\begin{equation}\label{julim}
U_{\theta_{h}\sigma}(t,s)=U_{\sigma}(t+h,s+h),\quad\forall\sigma\in\Sigma,\,t\geq s,\,h\in\mathbb{R}
\end{equation}
is satisfied. In this case, the parameter $\sigma$ is called the \textbf{symbol} of the process $\{U_{\sigma}(t,s):t\geq s\}$ and the set $\Sigma$ the \textbf{symbol space} of the system $\{ U_{\sigma}(t,s): t\geq s\}_{\sigma\in\Sigma}$.
\end{defin} 

If, for example, $\Xi = \mathcal{C}(\R, X)$ and $\sigma\in \Xi$, we can take $\theta$ as the shifting semigroup:
$$
\theta_{s}\sigma(\cdot)=\sigma(\cdot+s),\quad\forall s\in\mathbb{R},
$$
and we usually take $\Sigma$ to be the hull of a non-autonomous function $\sigma \in \mathcal{C}(\R, X)$, that is defined below:

\begin{defin}
Let $\Xi = \mathcal{C}(\R, X)$ be the space of continuous functions from $\R$ to $X$ endowed with the Fréchet metric $d_{\Xi}$:
$$
d_{\Xi}(\xi_{1},\xi_{2}):=\sum_{n=0}^{\infty}\frac{1}{2^{n}}\frac{d^{(n)}(\xi_{1},\xi_{2})}{1+d^{(n)}(\xi_{1},\xi_{2})},\qquad\xi_{1},\xi_{2}\in\Xi, 
$$
where
$$
d^{(n)}(\xi_{1},\xi_{2}):=\max_{s\in[-n,n]} \|\xi_1(s) - \xi_2(s)\|,\qquad n\in\mathbb{N}. 
$$

Notice that convergence in $(\Xi, d_\Xi)$ is equivalence to uniform convergence in compact sets of $\R$. 

We define the hull $\mathcal{H}_\Xi(g)$ of the non-autonomous function $g\in\Xi$ as
$$
\mathcal{H}_\Xi(g):={\overline{{\{\theta_r g:r\in\mathbb{R}\}}}}, 
$$
where $\theta_r g = g(\cdot + r)$ and the closure is taken under the metric $d_{\Xi}$ of $\Xi$.
\end{defin}

Therefore, if we want to study the Cauchy problem \eqref{intazim}, we can define $\Sigma = \mathcal{H}_\Xi(g)$ and the system of processes $\{ U_{\sigma}(t,s): t\geq s\}_{\sigma\in\Sigma}$ where $U_\sigma( \cdot,\tau) v_\tau : [\tau, \infty) \to X$ as the solution of problem \eqref{intmarjax} starting at $\tau$. If we have existence and uniqueness of solutions, then property \eqref{julim} follows readily.

\begin{defin}
For a system of processes $\left\{U_\sigma(t,s):t\geq s\right\}_{\sigma\in\Sigma}$, we can define its \textbf{uniform attractor} (if it exists) as the compact set $\mathcal{A}_\Sigma \subset X$ that is the minimal closed set with the property that for every bounded subset $B\subset X$, we have the uniform attraction property:
$$
\lim_{t\to \infty} \sup_{\sigma\in \Sigma} \dist_H(U_\sigma(t,0)B, \mathcal{A}_\Sigma) = 0
$$

We also say that a set $\B \subset X$ is a uniformly absorbing set for the system of processes $\left\{U_\sigma(t,s):t\geq s\right\}_{\sigma\in\Sigma}$ if for any $B \subset X$ bounded, there exists $t_B \geq 0$ such that:
$$
\{U_\sigma(t, 0) x: x\in B, t\geq t_B, \sigma \in \Sigma\} \subset \B.
$$
\end{defin}

Alternatively to the system of processes, we can study equation \eqref{intmarjax} with a semigroup:

\begin{defin}
    We say that a system of evolution processes $\{U_\sigma(t,s):t\geq s\}_{\sigma \in \Sigma}$ on the Banach space $X$ is $(\Sigma \times X, X)$-continuous if for every fixed $t$, $s \in \R$ with $t\geq s$, the map $(\sigma, u) \mapsto U_\sigma(t,s)u$ is continuous.
\end{defin}

\begin{defin}[Skew-product semigroup]
Let $\Sigma \subset \Xi$ be a symbol space and $\{U_\sigma(t,s):t\geq s\}_{\sigma \in \Sigma}$ be a $(\Sigma \times X, X)$-continuous system of evolution processes on the Banach space $X$. Let $\mathbb{X} = X\times \Xi$, we define the associated \textbf{skew-product} semigroup $\{S(t):t\geq 0\}$ by:
$$
S(t)(x, \sigma) = (U_\sigma(t,0)x, \theta_t \sigma),
$$
for all $t\geq 0$, $x\in X$, $\sigma \in \Sigma$.
\end{defin}

It has already been proved in literature (see \cite{booknolanovo}) that the uniform attractor of a system of evolution processes $\mathcal{A}_\Sigma$ can be identified with the projection on the first coordinate of the global attractor $\mathbb{A}$ of the associated skew-product semigroup (when these attractors exist).

We finish with a technical lemma that is not hard to prove.

\begin{lema}\label{Zungtown}
Let $\Sigma = \mathcal{H}_\Xi(g)$ for some $g\in \Xi$. If the application 
$(t, \sigma) \mapsto \theta_t \sigma$ is continuous, then: % This happens because every cont function is unif cont in compact set.
$$
\Sigma = \alpha(g) \cup \{\theta_t g : t\in \R\} \cup \omega(g).
$$
\end{lema}

% ####################################################################################################################################################################################################################################################################################################################################################################################################################################################################################################################################################################################################################################################################################################################################################################################################################################

\section{Exponential Uniform Attractors}\label{exp}

Let $\{U_\sigma(t,s)\}_{\sigma \in \Sigma}$ be a system of processes in $X$ with a uniformly absorbing set $\mathcal{B}$ and suppose that the following conditions are satisfied:

\begin{enumerate}[label=$(H_{\arabic*})$]
    \item \label{aux} There is an auxiliary Banach space $Y$ compactly embedded in $X$.
     \item \label{smoothing}  $\{U_\sigma(t,s)\}_{\sigma \in \Sigma}$ is uniformly smoothing on the set $\mathcal{B}$, that is, for any $t>0$, there exists $\kappa(t)$ such that
     $$
    \underset{\sigma\in \Sigma}{\sup} \|U_\sigma(t,0)x - U_\sigma(t,0)y\|_Y \leq \kappa(t) \|x-y\|_X, \quad \forall \, x,y \in \mathcal{B}.
     $$

     \item \label{holder} For $s\in \R$, $x$, $y\in \mathcal{B}$, we have:
     $$
     \| U_\sigma(t+s,s)x - U_\sigma(\Tilde{t}+s,s)y \| \leq C(r) (\|x-y\|^\gamma + |t-\Tilde{t}|^\theta),
    $$ for $t$, $\Tilde{t} \in [0,r]$ where $C(r)$ is uniform for $s \in \R$, $\sigma \in \Sigma$.
\end{enumerate}

Next theorem is based on the work on the paper \cite{efendiev}. Our result is an abstract version of their result that can be applied to other parabolic differential equation. We find an exponential uniform attractor of finite box-counting dimension for the process $\{U_\sigma(t,s):t\geq s\}$. Notice that this kind of result had already been found for global attractors and pullback attractors (see \cite{bispace} and \cite{estefanizona}), and that in \cite{cuietal}, a finite dimensional uniform attractor is found, but the uniform exponential attractor is not constructed.

\begin{theorem}[Exponential Uniform Attractors] \label{ralec}
    Let $\{U_\sigma(t,s)\}_{\sigma \in \Sigma}$ be a system of processes in $X$ satisfying hypothesis \ref{aux}-\ref{holder} and

    \begin{enumerate}[label=$(H_{\arabic*})$, resume]
    \item \label{lipdriving} The driving semigroup is Lipschitz for every $t\geq 0$. More precisely, there exist constants $P \geq 1$ and $\zeta > 0$ such that for all $t\geq 0$,
    $$
    \dist_\Xi(\theta_t \xi_1, \theta_t \xi_2) \leq P e^{\zeta t} \dist_\Xi(\xi_1, \xi_2), \quad \forall \xi_1, \, \xi_2\in \Xi.
    $$
    \item \label{lipschitz} $\{U_{\sigma}(t, s)\}_{\sigma \in \Sigma}$ is $(\Sigma,X)$-Lipschitz on the absorbing set $\mathcal{B}$, satisfying
	\begin{equation}\label{lipschitzsymbol}
		\| U_{\sigma_{1}}(t,0) x - U_{\sigma_{2}}(t,0) x \|_X \leqslant L(t) d_\Xi (\sigma_{1} , \sigma_{2} ), \qquad \forall t \geq 1, \, \sigma_{1}, \sigma_{2} \in \Sigma, \, x \in \mathcal{B}, 
	\end{equation}
	where $1\leq L(t) \leq c_1 e^{\beta t}$ for some positive constants $ c_1 ,\beta > 0$ for $t \geq 1$.
    \item The symbol space $\Sigma \subset \Xi$ has finite box counting dimension $d_\Sigma$ over $\Xi$.
    \end{enumerate}

 Moreover, let $\tau$ be the time when $\mathcal{B}$ absorbs itself, that is:
    \begin{equation}\label{fatal}
    \bigcup_{\sigma\in \Sigma} U_\sigma(t,0)\mathcal{B} \subset \B , \quad \forall \, t\geq \tau.        
    \end{equation}

Then, for each $\nu\in (0,1)$, there exits an exponential attractor $\mathbb{M}_\nu$ for the skew-product semigroup $\{S(t):t\geq 0\}$, such that for any bounded set $\mathbb{D} \subset \mathbb{X}$, there exists a constant $C(\mathbb{D}) \geq 1$ such that:
$$
\dist_H(S(t)\mathbb{D}, \mathbb{M}_\nu) \leq C(\mathbb{D}) e^{- |\log(\nu)| \frac{\gamma}{\tau}t}, \quad \forall \, t\geq 0
$$
and this exponential attractor has bounded box-counting dimension, satisfying the estimate:
$$
d_B(\mathbb{M}_\nu) \leq   \frac{1}{\theta} + \frac{1}{\gamma} \left [ \frac{\log N_{\frac{\nu}{2\kappa(\tau)}}}{-\log \nu} + d_\Sigma \left ( \frac{(\beta + \zeta) \tau}{-\log \nu} +1 \right ) \right ],
$$
where $N_{\frac{\nu}{2\kappa(\tau)}}$ is the number of balls of radius $\frac{\nu}{2\kappa(\tau)}$ in $X$ necessary to cover the unit ball $B_Y(0,1)$ in $Y$.

The projected set $\mathcal{M}_\nu = \Pi_1 \mathbb{M}_\nu \subset X$ is an uniform exponential attractor for the system of processes $\{U_\sigma(t,s): t\geq s\}_{\sigma\in \Sigma}$, satisfying the same estimate for box-counting dimension:
\begin{equation}\label{lajym}
  d_B(\mathcal{M}_\nu) \leq   \frac{1}{\theta} + \frac{1}{\gamma} \left [ \frac{\log N_{\frac{\nu}{2\kappa(\tau)}}}{-\log \nu} + d_\Sigma \left ( \frac{(\beta + \zeta) \tau}{-\log \nu} +1 \right ) \right ].  
\end{equation}

Finally, if $D\subset X$ is bounded, there exists a constant $C(D)$ such that:
$$
\sup_{\sigma\in \Sigma} \dist_H(U_\sigma(t,s) D, \mathcal{M}_\nu) \leq C(D) e^{- |\log(\nu)| \frac{\gamma}{\tau}(t-s)}, \quad \forall \, t\geq s.
$$
    
\end{theorem}

\begin{proof}

\textit{First part:}
Let $B$ be the closed bounded uniformly absorbing set of the system $\{U_\sigma(t,s):t\geq s\}_{\sigma\in \Sigma}$ and consider the lifted set $\mathbb{B} = B\times \Sigma$.

Then it is easy to see that if $\{S(t): t\in \R\}$ is the skew-product semigroup generated by the system $\{U_\sigma(t,s)\}_{\sigma\in \Sigma}$, then $S(\tau) \mathbb{B} \subset \mathbb{B}$.

For all $n\in \N$, consider the covering of the finite dimensional set $\Sigma$ with a minimum number of open balls of $\Xi$ of centers $\{\Sigma_-^n, \dots, \sigma_N^n\}$ and radius $R_n = R\nu^n/P e^{\zeta n \tau} L(n\tau)$. We name these balls $B_i^n = B_\Xi(\sigma_i^n, R_n)$. The number $N = N_n$ depends on the box counting dimension of $\Sigma$. More precisely, for every $\delta > 0$, there exists $n_0 = n_0(\delta)$ such that
\begin{equation}
    N_n \leq \left ( \frac{Pe^{\zeta n \tau} L(n \tau)}{R \nu^n} \right )^{d+\delta}, \quad \forall \, n\geq n_0
\end{equation}

    Let $\nu \in (0,1)$ and $R> 0$, $x_0 \in \B$ be such that $\mathcal{B} = B_X(x_0, R) \cap \mathcal{B}$. 

    Let $\kappa = \kappa(\tau)$. Since $Y$ is compactly embedded in $X$, the unit ball $B_Y(0,1)$ in $Y$ can be covered by $N_{\frac{\nu}{2\kappa}}$ balls of radius $\frac{\nu}{2\kappa}$ in $X$, that is:
    \begin{equation}\label{gabrael}
B_Y(0,1) \subset \bigcup_{i=1}^{N_{\frac{\nu}{2\kappa}}} B_X\left (x_i, \frac{\nu}{2\kappa} \right ), \quad x_i \in B_Y(0,1).
    \end{equation}

Now, let $\sigma \in \Sigma$ be arbitrary, it follows from \ref{smoothing} that:
\begin{equation}\label{buchin}
\begin{split}
    U_\sigma(\tau,0) \B & = U_\sigma(\tau,0) [B_X(x_0, R) \cap \mathcal{B}] \\
    & \subset B_Y(U_\sigma(\tau,0) x_0, \kappa R) \cap U_\sigma(\tau, 0) \B
\end{split}
\end{equation}

Let $y_\sigma = U_\sigma(\tau,0)x_0$. Then it follows from \eqref{fatal}, \eqref{gabrael} and \eqref{buchin} that:
\begin{equation}
    U_\sigma(\tau,0) \B \subset \bigcup_{i=1}^{N_{\frac{\nu}{2\kappa}}} B_X\left ( y_\sigma + \kappa R x_i,   \frac{R \nu}{2}\right ) \cap \B,     
\end{equation}  
so there are $q^\sigma_i \in \B$, $i = 1, \cdots, N_{\frac{\nu}{2\kappa}}$ such that:
\begin{equation}
    U_\sigma(\tau,0) \B \subset \bigcup_{i=1}^{N_{\frac{\nu}{2\kappa}}} B_X\left ( q^\sigma_i, R\nu  \right ) \cap \B,     
\end{equation}  

Suppose by induction that for some $k \geq 1$, we have $N_X(U_\sigma(k \tau,0) \B, R\nu^k) \leq N_{\frac{\nu}{2\kappa}}^k$ (we just proved that this is true for $k = 1$). Then, there exist $q_i^\sigma \subset \B$, $i = 1, \cdots, N_{\frac{\nu}{2\kappa}}^k$ such that:
\begin{equation}\label{fagundes}
U_\sigma(k \tau,0) \B \subset \bigcup_{i=1}^{N_{\frac{\nu}{2\kappa}}^k} B_X\left ( q^\sigma_i, R\nu^k  \right ) \cap \B,     
\end{equation}

Now, by the system of processes property:
\begin{equation*}
\begin{split}
U_\sigma((k+1) \tau,0) \B & = U_\sigma((k+1) \tau, k \tau) U_\sigma(k \tau, 0) \B \\
& \subset U_{\theta_{k \tau}\sigma}(\tau,0) U_\sigma(k \tau, 0) \B
\end{split}
\end{equation*}

Then, by \eqref{fagundes}, 
\begin{equation}\label{bazax}
U_\sigma((k+1) \tau,0) \B \subset \bigcup_{i=1}^{N_{\frac{\nu}{2\kappa}}^k} U_{\theta_{k \tau}\sigma}(\tau,0) [B_X\left ( q^\sigma_i, R\nu^k  \right ) \cap \B]    
\end{equation}
and for each $i$, we have by the smoothing property \ref{smoothing} and self absorption of $\B$ that:
\begin{equation}\label{matrip}
\begin{split}
U_{\theta_{k \tau}\sigma}(\tau,0) [B_X\left ( q^\sigma_i, R\nu^k  \right ) \cap \B] & \subset B_Y(U_{\theta_{k \tau}\sigma}(\tau,0) q^\sigma_i, \kappa R \nu^k) \cap \B \\
& \subset \bigcup_{j=1}^{N_{\frac{\nu}{2\kappa}}} B_X\left ( p_{i,j}, R\nu^{k+1}  \right ) \cap \B,     
\end{split}
\end{equation}
for $p_{i, j} \in \B$. Then, using \eqref{bazax} and \eqref{matrip}, we get $N_X(U_\sigma((k+1) \tau,0) \B, R\nu^{k+1}) \leq N_{\frac{\nu}{2\kappa}}^{k+1}$. 

Then, we proved by induction that for any $\sigma \in \Sigma$, 
\begin{equation}\label{mutum}
N_X(U_\sigma(n \tau,0) \B, R\nu^n) \leq N_{\frac{\nu}{2\kappa}}^n, \quad \forall \, n\in \N.
\end{equation}

For each $\sigma \in \Sigma$, chose a set of centers $W^n(\sigma) \subset X$ such that $\# W^n(\sigma) \leq N^n_{\frac{\nu}{2\kappa}}$ and:
\begin{equation}
    U_\sigma(n\tau, 0) \mathcal{B} \subset \bigcup_{u\in W^n(\sigma)} B_X(u, R\nu^n) \cap \mathcal{B}
\end{equation}

Now, for each $n\in \N$ and $j\in \{1, \dots N(n)\}$, we define:
$$
\mathcal{U}^j(n) = W^n(\sigma_j^n),
$$
$$
\mathcal{U}(n) = \bigcup_{j=1}^{N(n)} \mathcal{U}^j(n).
$$

\textit{Second part:}

It follows from the definition of skew-product semigroup that if $\Pi_1 : X\times \Xi \to X$ is the projection onto the first coordinate, then: 
$$
\Pi_1 S(n\tau) \mathbb{B} = U_\Sigma(n\tau, 0) B
$$

Next we will show that:
\begin{equation}\label{gropa}
\dist_H\left ( \Pi_1 S(n\tau) \mathbb{B}, \mathcal{U}(n) \right ) \leq 2R \nu^n.
\end{equation}
Indeed, if $y\in \Pi_1 S(n\tau) \mathbb{B}$, then $y = U_\sigma(n\tau,0)x$, for some $\sigma \in \Sigma$, $x\in B$.

By the covering of $\Sigma$, there exists $i\in \{1, \dots, N(n)\}$ such that
$$
d_\Xi(\sigma, \sigma_i^n) \leq \frac{R \nu^n}{L(n \tau)}.
$$
And by the definition of the sets $\mathcal{U}^i(n)$, we have:
\begin{equation}\label{abrop1}
    d\left ( U_{\sigma_i^n}(n\tau,0) x, \mathcal{U}^i(n) \right ) \leq R\nu^n.
\end{equation}
Moreover, by \ref{lipschitz}, we have:
\begin{equation}\label{abrop2}
    \|y - U_{\sigma_i^n}(n\tau,0) x\| \leq L(n\tau) \frac{R \nu^n}{L(n\tau)} = R\nu^n
\end{equation}
And the estimates \eqref{abrop1} and \eqref{abrop2} imply \eqref{gropa}. Notice that, more precisely, we showed that:
\begin{equation}\label{braaaaax}
\dist_H\left ( U_{B_i^n}(n\tau, 0)B, \mathcal{U}^i(n) \right ) \leq 2R \nu^n
\end{equation}

Now, for each $v\in \mathcal{U}^i(n)$ take, if existing, $(u_v, \xi_v) \in \mathbb{B}$ such that $\xi_v\in B_i^n$ and 
\begin{equation}
    \|v - U_{\xi_v}(n\tau, 0) u_v\| \leq 2R\nu^n.
\end{equation}
If there is no such pair $(u_v, \xi_v)$, then we can take off the element $v$ from $\mathcal{U}^i(n)$ still preserving the estimate \eqref{braaaaax}.

Now we define $\mathbb{U}^i(n) = \{S(n\tau) (u_v, \xi_v) : v\in \mathcal{U}^i(n)\}$, and
$$
\mathbb{U}(n) = \bigcup_{i=1}^{N(n)}\mathbb{U}^i(n).
$$
Now, using \eqref{gropa} and the definition of $\mathbb{U}(n)$, it is not hard to see that:
$$
\dist_H \left ( \Pi_1 S(n\tau) \mathbb{B}, \Pi_1 \mathbb{U}(n) \right ) \leq 4R\nu^n.
$$

We also notice that $\mathbb{U}(n) \subset S(n\tau) \mathbb{B}$ and 
$$
\# \mathbb{U}(n) \leq \# \mathcal{U}(n) \leq N(n) N^n_{\frac{\nu}{2\kappa}}.
$$

\textit{Third part:}

We define now the sets:
$$
\mathbb{E}(0) = \mathbb{U}(0), \quad \mathbb{E}(k+1) = S(\tau) \mathbb{E}(k) \cup \mathbb{U}(k+1)
$$

The set $\mathbb{M}^d : = \overline{\bigcup_{k\in \mathbb{N}} \mathbb{E}(k)}$ will be the exponential attractor for the discrete semigroup $\{S(n\tau): n\in \N\}$.

First, notice that $S(\tau) \mathbb{M}^d \subset \mathbb{M}^d$. Now if $n\in \N$, we have:
$$
\dist_H(S(n\tau) \mathbb{B}, \mathbb{M}^d) \leq  \underset{(x,\sigma) \in \mathbb{B}}{\sup} \dist[(U_\sigma(n\tau, 0) x, \theta_{n\tau} \sigma), \mathbb{U}(n)]
$$

For  each $(x,\sigma) \in \mathbb{B}$, there exists a $i\in \{1, \dots, N(n)\}$ such that $\sigma \in B_i^n$, and $v\in \mathcal{U}^i(n)$ such that
$$
\|U_\sigma(n\tau,0) x - v\| \leq 2R\nu^n,
$$
and there exist $(u_v, \xi_v)$ such that $\xi_v \in B^n_i$, and 
$$
\|v - \Pi_1 S(n\tau) (u_v, \xi_v)\| \leq 2R\nu^n
$$

And we have $u_n = S(n\tau) (u_v, \xi_v) \in \mathbb{U}(n)$ and:
\begin{equation}
\begin{split}
\dist[(U_\sigma(n\tau, 0) x, \theta_{n\tau} \sigma), u_n] & \leq \|U_\sigma(n\tau,0) x - \Pi_1 S(n\tau) (u_v, \xi_v)\| + \dist_\Xi(\theta_{n\tau}\sigma, \theta_{n\tau} \xi_v)  \\
& 4R\nu^n + Pe^{\zeta n\tau} 2 \frac{R\nu^n}{Pe^{\zeta n \tau} L(n\tau)} \leq 6R\nu^n
\end{split}
\end{equation}
which implies that 
\begin{equation}\label{agron}
    \dist_H(S(n\tau) \mathbb{B}, \mathbb{M}^d) \leq 6R\nu^n
\end{equation}
which gives the exponential attraction of $\mathbb{B}$ by $\mathbb{M}^d$ under the action of $\{S(n\tau): n\in \N\}$. Since $\mathbb{B}$ is an absorbing set for the semigroup, the exponential attraction extends to every bounded set $\mathbb{D} \subset \mathbb{X}$.

\textit{Fourth part:}
Now let us prove that $\mathbb{M}^d$ has finite box-counting dimension. Since $\mathbb{U}(n) \subset S(n\tau)\mathbb{B}$, it can be shown by induction that $\mathbb{E}((n+j)\tau) \subset S((n+j)\tau) \mathbb{B} \subset S(n\tau) \mathbb{B}$, which implies:
\begin{equation}
\mathbb{M}^d \subset \bigcup_{j=1}^{n-1} \mathbb{E}(n) \cup \overline{S(n\tau) \mathbb{B}}
\end{equation}
It follows from the estimate on the cardinality of the sets $\mathbb{U}(n)$, $n\in \N$ and the construction of the sets $\mathbb{E}(n)$, $n\in \N$, that:
$$
\# \mathbb{E}_0 \cup \cdots \cup \mathbb{E}_{n-1} \leq \frac{n(n+1)}{2} N(n) N^n_{\frac{\nu}{2\kappa}}
$$

It has been proved before that
$$
\dist_H(S(n\tau) \mathbb{B}, \mathbb{U}(n)) \leq 6 R\nu^n
$$
Which implies that:
$$
N(S(n\tau) \mathbb{B}, 7R\nu^n) \leq  \# \mathbb{U}(n) \leq N(n) N^n_{\frac{\nu}{2\kappa}}
$$

Then, we have:
$$
N(\mathbb{M}^d, 7R\nu^n) \leq \frac{n^2+n+2}{2} N(n) N^n_{\frac{\nu}{2\kappa}}
$$
Which implies that 
\begin{equation}
\begin{split}
d_B(\mathbb{M}^d) \leq \underset{n\to \infty}{\limsup} \frac{\log \left ( \frac{n^2+n+2}{2} N(n) N^n_{\frac{\nu}{2\kappa}} \right )}{- \log (7R\nu^n)}.
\end{split}
\end{equation}
\end{proof}
To estimate this, we notice that
$$
\underset{n\to \infty}{\lim} \frac{\log \frac{n^2+n+2}{2}}{-\log(7R\nu^n)} = 0
$$
$$
\underset{n\to \infty}{\lim} \frac{\log N^n_{\frac{\nu}{2\kappa}}}{-\log(7R\nu^n)} = \frac{\log N_{\frac{\nu}{2\kappa}}}{-\log \nu}
$$

Finally, by definition of $N(n)$, we have:
\begin{equation*}
\underset{n\to \infty}{\limsup} \frac{\log N(n)}{-\log(7R\nu^n)} \leq \underset{n\to \infty}{\limsup} \frac{\log N_\Xi\left ( \Sigma, \frac{R}{Pc_1} \left ( \frac{\nu}{e^{(\zeta + \beta) \tau}} \right )^n\right )}{-\log (7R \nu^n)} \leq d_\Sigma \left ( \frac{(\beta + \zeta) \tau}{-\log \nu} +1 \right ) 
\end{equation*}
where we used the box-counting dimension $d_\Sigma$ of $\Sigma$ in $\Xi$.

Then, we conclude that:
$$
d_B(\mathbb{M}^d) \leq \frac{\log N_{\frac{\nu}{2\kappa}}}{-\log \nu} + d_\Sigma \left ( \frac{(\beta + \zeta) \tau}{-\log \nu} +1 \right ).
$$

\textit{Fifth part:} Now, we will extend the exponential attractor $\mathbb{M}^d$ to an exponential attractor of the continuous semigroup $\{S(t): t\geq 0\}$. We define:
$$
\mathbb{M} = \bigcup_{t\in [0,\tau]} S(t) \mathbb{M}^d.
$$

The positive semi-invariance of $\mathbb{M}$ by $\{S(t):t\in \R\}$ follows from the positive semi-invariance of $\mathbb{M}^d$ by $\{S(n\tau): n\in \N\}$.

For the exponential attraction, let $t\geq 0$ and we write $t = n\tau + s$ with $s\in [0,\tau)$. Using \ref{holder}, \ref{lipdriving} and \ref{lipschitz} we can show that $S(t)$ is $\gamma$-Hölder continuous in $\mathcal{B} \times \Sigma$ with a constant we call $K(t)$, and $K(t)$ can be took uniform $(=K)$ for $t\in [0,\tau]$. Then we have:
\begin{equation}
\begin{split}
\dist_H(S(t)\mathbb{B}, \mathbb{M}) & \leq \dist_H(S(s) S(n\tau) \mathbb{B}, S(s) \mathbb{M}^d) \leq K (\dist_H(S(n\tau) \mathbb{B}, \mathbb{M}^d))^\gamma \\
& \leq K (6R\nu^n)^\gamma \leq C e^{\log(\nu)\gamma \frac{t}{\tau}} 
\end{split}
\end{equation}

This implies that $\mathbb{M}$ attracts any bounded set in $\mathbb{X}$ exponentially under the action of $\{S(t):t\in \mathbb{R}\}$.

Finally, we estimate the box-counting dimension of $\mathbb{M}$. Notice that:
$$
\mathbb{M} = \Phi([0,\tau] \times \mathbb{M}^d),
$$
where $\Phi: [0,\tau] \times \mathbb{M}^d \to \mathbb{M}$ is given by $\Phi(t,(x,\sigma)) = S(t) (x,\sigma)$. Since $\Phi$ is $\theta$-Hölder continuous in time and $\gamma$-Hölder continuous in the $\mathbb{X}$ variable, it follows that:
$$
d_B(\mathbb{M}) \leq \frac{1}{\theta} + \frac{1}{\gamma} d_B(\mathbb{M}^d) \leq   \frac{1}{\theta} + \frac{1}{\gamma} \left [ \frac{\log N_{\frac{\nu}{2\kappa}}}{-\log \nu} + d_\Sigma \left ( \frac{(\beta + \zeta) \tau}{-\log \nu} +1 \right ) \right ].
$$

% ####################################################################################################################################################################################################################################################################################################################################################################################################################################################################################################################################################################################################################################################################################################################################################################################################################################

\section{Main Results}\label{main}

Finally, we prove Theorem \ref{intdenras}, which estimates the box-counting dimension of the union of pullback attractors using the semi-continuity of these attractors as $t\to -\infty$ and $t\to \infty$.

\begin{proof}[Proof of  Main Theorem \ref{intdenras}]

By the first part of the proof of Theorem \ref{ralec}, for any $\sigma \in \Sigma$, 
\begin{equation}\label{mutam}
N_X(U_\sigma(n \tau,0) \B, R\nu^n) \leq N_{\frac{\nu}{2\kappa}}^n, \quad \forall \, n\in \N.
\end{equation}

If $\sigma\in \Sigma$, using the fact that $D = \bigcup_{t\in \R}A(t)$ is bounded, it is not hard to prove that $D\subset \B$. For $i\in \Z$, we have:
\begin{equation}
\begin{split}
A(i  \tau) & = U_\psi(i  \tau, (i-n)\tau) A((i-n)\tau) \\
& \subset U_{\theta_{(i-n)\tau}\psi}(n\tau, 0) \B
\end{split}
\end{equation}

Therefore, by \eqref{mutam}, we have:
\begin{equation}\label{ukrux}
    N(A(i\tau), R\nu^n) \leq N_{\frac{\nu}{2\kappa}}^n, \text{ for all } n\in \N, i\in \Z.
\end{equation}

Let $C = C(\tau)$ from \ref{holder} and $\Bar{R}= C R^{\gamma}$, and let $n \in \N$ be large enough so that there exists $T_n \geq \Tilde{T}$ such that $f(T_n) = \Bar{R}\nu^n$. By \ref{intentering} and the fact that $f$ is strictly decreasing, we get:
$$
\bigcup_{t\geq T_n} A(t) \subset B_X(A_+, \Bar{R}\nu^n)
$$ and $$
\bigcup_{t\leq - T_n} A(t) \subset B_X(A_-, \Bar{R}\nu^n)
$$

This implies that 
$$
N_X\left (\bigcup_{t\geq T_n} A(t), 2\Bar{R}\nu^n\right ) \leq N_X(A_+, \Bar{R}\nu^n)
$$ and $$
N_X\left (\bigcup_{t\leq - T_n} A(t), 2\Bar{R}\nu^n\right ) \leq N_X(A_-, \Bar{R}\nu^n).
$$

Now let us estimate the number of balls of radius $2\Bar{R}\nu^n$ we need to cover $\bigcup_{t\in [- T_n,T_n]} A(t)$. First notice that if $\lfloor r \rfloor$ and $\lceil r \rceil$ denote the nearest integers to $r$ from below and above, respectively, we have:
$$
\bigcup_{t\in [- T_n ,T_n]} A(t) \subset \bigcup_{i= - \lceil T_n/\tau \rceil }^{\lfloor T_n/\tau \rfloor}  \bigcup_{s\in [i, i+1]}A(s \tau)
$$
Now we define $\Phi_i:[0,1]\times A(i\tau)$ by:
$$
\Phi_i(t,x) = U_\psi((t+i)\tau,i\tau)x.
$$
Then, since $U_\psi(t,s)A(s) = A(t)$ for all $t\geq s$, we conclude:
$$
\bigcup_{s\in [i, i+1]}A(s \tau) = \Phi_i\left ( [0,1] \times A(i \tau) \right ).
$$
Now we cover $[0,1]$ with $N_1 = \big \lfloor \frac{2}{\left ( \frac{\Bar{R}\nu^n}{C\tau^\theta}\right )^{1/ \theta}} \big \rfloor$ balls of radius $\left ( \frac{\Bar{R}\nu^n}{C\tau^\theta}\right )^{1/ \theta}$ centered at $t_j\in [0,1]$, $j = 1, \cdots, N_1$, and we also cover $A(i\tau)$ using \eqref{ukrux}:
$$
A(i\tau) = \bigcup_{k=1}^{N_{\frac{\nu^{1/\gamma}}{2\kappa}}^n} B(x_k, R\nu^{n/\gamma}) \cap A(i\tau)
$$
with $x_k \in A(i\tau)$.

Then:
\begin{equation*}
\begin{split}
\bigcup_{s\in [i, i+1]}A(s \tau) & \subset \Phi_i\left ( \bigcup_{j=1}^{N_1} B_\R \left (t_j, \left ( \frac{\Bar{R}\nu^n}{C\tau^\theta}\right )^{1/ \theta} \right ) \cap [0,1] \times \bigcup_{k=1}^{N_{\frac{\nu^{1/\gamma}}{2\kappa}}^n} B(x_k, R\nu^{n/\gamma})\cap A(i\tau) \right )\\
& \subset \bigcup_{j=1}^{N_1} \bigcup_{k=1}^{N_{\frac{\nu^{1/\gamma}}{2\kappa}}^n} \Phi_i \left ( B_\R\left (t_j, \left ( \frac{\Bar{R}\nu^{n}}{C\tau^\theta}\right )^{1/ \theta} \right ) \cap [0,1] \times B(x_k, R\nu^{n/\gamma})\cap A(i\tau) \right ) \\
& \subset \bigcup_{j=1}^{N_1} \bigcup_{k=1}^{N_{\frac{\nu^{1/\gamma}}{2\kappa}}^n} B_X(U_\psi(t_j+i)\tau, i\tau)x_k, 2\Bar{R} \nu^n)
\end{split}
\end{equation*}

Therefore, we can conclude that:
\begin{equation*}
\begin{split}
N_X\left (\bigcup_{t\in \R} A(t), 2\Bar{R}\nu^n \right) & \leq N_X(A_+,\Bar{R}\nu^n) + N_X(A_-, \Bar{R}\nu^n) + \left ( 2 \frac{T_n}{\tau} + 3 \right ) \cdot N_1 \cdot N_{\frac{\nu^{1/\gamma}}{2\kappa}}^n \\
& \leq 3 \cdot \max \left ( N_X(A_+,\Bar{R}\nu^n), N_X(A_-, \Bar{R}\nu^n), \left ( 2 \frac{T_n}{\tau} + 3 \right ) \cdot N_1 \cdot N_{\frac{\nu^{1/\gamma}}{2\kappa}}^n \right )
\end{split}
\end{equation*}

Therefore, 
\begin{equation*}
\begin{split}
d_B\left ( \bigcup_{t\in \R} A(t) \right ) & = \limsup_{n\to \infty} \frac{\log N_X\left ( \bigcup_{t\in \R} A(t), 2\Bar{R}\nu^n \right )}{- \log 2\Bar{R}\nu^n} \leq  \max \left (L_1,  L_2, L_3 \right )
\end{split}
\end{equation*}
where $$
L_1 = \limsup_{n\to \infty} \frac{\log N_X(A_+,\Bar{R}\nu^n)}{- \log 2\Bar{R}\nu^n} = d_+
$$
$$ L_2 = \limsup_{n\to \infty} \frac{\log N_X(A_-, \Bar{R}\nu^n)}{- \log 2\Bar{R}\nu^n} = d_-$$
$$L_3 = \limsup_{n\to \infty} \frac{\log \left [ \left ( 2 \frac{T_n}{\tau} + 3 \right ) \cdot N_1 \cdot N_{\frac{\nu^{1/\gamma}}{2\kappa}}^n \right ]}{- \log 2\Bar{R}\nu^n}$$

If $f(t) = K / t^r$, for some $K \geq 1$ and $r\in (0,\infty)$, we have $T_n = \left (\frac{K}{\Bar{R}\nu^n} \right )^{\frac{1}{r}}$, and we conclude that:
\begin{equation*}
\begin{split}
L_3 & \leq \limsup_{n\to \infty} \frac{\log \left ( 2 \frac{T_n}{\tau} + 3 \right )}{- \log 2\Bar{R}\nu^n} + \limsup_{n\to \infty} \frac{\log \left [ \frac{2}{\left ( \frac{\Bar{R}\nu^n}{C\tau^\theta}\right )^{1/ \theta}} \right ]}{- \log 2\Bar{R}\nu^n} + \limsup_{n\to \infty} \frac{\log N_{\frac{\nu^{1/\gamma}}{2\kappa}}^n}{- \log 2\Bar{R}\nu^n} \\
& = \frac{1}{r} + \frac{1}{\theta} + \frac{\log N_{\frac{\nu^{1/\gamma}}{2\kappa}}}{-\log \nu}
\end{split}    
\end{equation*}

Taking $\nu = (1/2)^\gamma$, we get:
$$
L_3 \leq \frac{1}{r} + \frac{1}{\theta} + \frac{1}{\gamma}\cdot \log_2 N_{\frac{1}{4\kappa}}
$$

This proves \ref{intT1}, and \ref{intT2} follows from \ref{intT1}. 
\end{proof}

Next proposition gives sufficient conditions for semi-continuity of pullback attractors \eqref{zalux} and \eqref{zalux2}. The conditions are exponential convergence of symbols by translation, and the existence of exponential uniform attractors for the symbols in $\Sigma_-$ and $\Sigma_+$.

\begin{prop}\label{barol}
Let $\{U_\sigma(t,s):t\geq s\}_{\sigma\in \Sigma}$ be a system of evolution processes in the metric space $X$, for $\Sigma = \mathcal{H}_\Xi(g)$, and suppose that $\{U_g(t,s): t\geq s\}$ has a compact pullback attractor $D = \bigcup_{t\in \R}A(t)$. Suppose that the system satisfies \ref{lipschitz}, and that there are compact invariant sets $\Sigma_-$ and $\Sigma_+\subset \Xi$ that satisfy the following attraction properties:
$$
\dist(\theta(t) g, \Sigma_-) \leq C e^{\eta_1 t},
$$
$$
\dist(\theta(t) g, \Sigma_+) \leq C e^{-\eta_2 t}.
$$

Moreover, suppose that there exist uniform exponential attractors $\mathcal{M}_{\Sigma_-}$ and $\mathcal{M}_{\Sigma_+}$ such that there exist constants $\nu_2 > 0$ and $\nu_1 > 0$ and $C(D) \geq 1$ such that:
$$
\sup_{\sigma\in \Sigma_-} \dist_H(U_\sigma(t,0) D, \mathcal{M}_{\Sigma_-}) \leq C(D) e^{-\nu_1 t}.
$$
$$
\sup_{\sigma\in \Sigma_+} \dist_H(U_\sigma(t,0) D, \mathcal{M}_{\Sigma_+}) \leq C(D) e^{-\nu_2 t},
$$

Then we have:
\begin{equation}\label{zalux2}
\dist_H(A(T), \mathcal{M}_{\Sigma_-}) \leq \Bar{c} e^{\xi T}, \quad T\leq - \Tilde{T},
\end{equation}
\begin{equation}\label{zalux}
\dist_H(A(T), \mathcal{M}_{\Sigma_+}) \leq \Bar{c} e^{- \xi T}, \quad T\geq \Tilde{T},
\end{equation}
where $\Bar{c} \geq 0$, $\xi > 0$ and $\Tilde{T} \geq 0$ are constants.
\end{prop}

\begin{proof}
We will prove only \eqref{zalux}, and the proof of \eqref{zalux2} is analogous. Let $T\geq 0$ and $t \geq 0$, then, for each $s\geq 0$, take $\sigma(s) \in \Sigma_+$ such that 
$$
\dist(\theta_s g, \sigma(s)) \leq 2 C e^{-\eta_2 s}
$$
Then, we have:
\begin{equation*}
\begin{split}
\dist_H(A(T), \mathcal{M}_{\Sigma_+}) & \leq \sup_{x\in A(T-t)} \dist(U_{\theta_{T-t}g}(t,0)x, U_{\sigma(T-t)}(t,0)x) \\
& + \sup_{x\in A(T-t)} \dist(U_{\sigma(T-t)}(t,0)x, \mathcal{M}_{\Sigma_+}) \\
& \leq 2C e^{-\eta_2(T-t)} c_1  e^{\beta t} + C(D) e^{-\nu_2 t} = Qe^{-\eta_2 T}e^{(\beta+\eta_2)t} +C(D) e^{-\nu_2t}
\end{split}
\end{equation*}
Let $c = C(D)$. If $\epsilon\in (0,c)$, we take $t_\epsilon = \log\left ( \frac{c}{\epsilon}\right )^{\frac{1}{\nu_2}}$, and we have:
\begin{equation}\label{gazarol}
\dist_H(A(T), \mathcal{M}_{\Sigma_+})\leq Qe^{-\eta_2 T}\left (  \frac{c}{\epsilon}\right )^{\frac{\beta+ \eta_2}{\nu_2}} +\epsilon.
\end{equation}

To minimize the function on the left hand side, we take:
$$
\epsilon_0 = c^{\frac{\beta+ \eta_2}{\beta+ \eta_2 +\nu_2}}\left ( \frac{\beta+ \eta_2}{\nu_2}\right )^{\frac{\nu_2}{\nu_2 +\beta+ \eta_2}}  (Qe^{-\eta_2 T})^{\frac{\nu_2}{\nu_2 + \beta+ \eta_2}}
$$

For $T$ large enough (say $T\geq \Tilde{T}$ for a large $\Tilde{T}$), $\epsilon_0 < c$. Using \eqref{gazarol} with $\epsilon = \epsilon_0$, we get:
$$
\dist_H(A(T), \mathcal{M}_{\Sigma_+}) \leq \Bar{c} e^{-\frac{\eta_2 \nu_2}{\nu_2 + \beta+ \eta_2} T}, \quad T\geq \Tilde{T}
$$
where $\Bar{c}>0$ is a constant.
\end{proof}

Finally, we prove Theorem \ref{inthuuuuglou}, which ties our results together, providing a way to estimate the box-counting dimension of the uniform attractor of dynamical systems without assuming that their symbol space is finite dimensional.

\begin{proof}[Proof of  Main Theorem \ref{inthuuuuglou}]
First, it follows from Lemma \ref{Zungtown} that:
$$
\mathcal{A}_{\Sigma} = \bigcup_{\sigma\in \Sigma}A_\sigma(0) \subset \mathcal{A}_{\Sigma_-} \cup \mathcal{A}_{\Sigma_+} \cup \bigcup_{t\in \R} A_g(t)
$$

Now, Theorem \ref{ralec} implies that the systems $\{U_\sigma(t,s): t\geq s\}_{\sigma\in \Sigma_-}$ and $\{U_\sigma(t,s): t\geq s\}_{\sigma\in \Sigma_+}$ have finite dimensional exponential uniform attractors $\mathcal{M}_{\Sigma_-} \supset \mathcal{A}_{\Sigma_-}$ and $\mathcal{M}_{\Sigma_+}\supset \mathcal{A}_{\Sigma_+}$ by minimality of the uniform attractor. Then we can apply Proposition \ref{barol} to show that the pullback attractors $\{A_g(t): t\in \R\}$ satisfy \ref{intentering} with an exponential function $f(t)$ and sets $\mathcal{A}_- = \mathcal{M}_{\Sigma_-}$ and $\mathcal{A}_+ = \mathcal{M}_{\Sigma_+}$. 

Then, we can estimate:
$$
d_B(\mathcal{A}_\Sigma) \leq \max\left ( d_B(\mathcal{M}_{\Sigma_-}), d_B(\mathcal{M}_{\Sigma_+}), d_B\left(\bigcup_{t\in \R} A_g(t) \right ) \right )
$$

This, along with \ref{intT2} and \eqref{lajym}, gives the desired estimate \eqref{intatroz}.
\end{proof}

\section{Applications}\label{app}

Now we study the reaction-diffusion problem \eqref{intazim}
%\begin{equation} \label{intazim}
%\begin{split}
%&v_t  = \Delta v + f( v )+g(t),\quad x\in\mathcal{O},\ t>\tau,\\
%&v (x,\tau)= v _{\tau}(x),\quad x\in\mathcal{O},\\
%&v (x,t)=0,\quad x\in\partial\mathcal{O},\ t\geqslant 0,
%\end{split} 
%\end{equation}
%\noindent where ${\mathcal{O}}\subset\mathbb{R}^{N}$, ($N\in\mathbb{N}$) is a bounded smooth domain. 
Let $X = L^2(\O)$ and $\Xi:={\mathcal{C}}(\mathbb{R},L^2(\O)),$ the space of all continuous functions $\xi:\mathbb{R} \rightarrow L^2(\O),$ endowed with the Fréchet metric $d_{\Xi}$:
\begin{equation}\label{albatroz}
d_{\Xi}(\xi_{1},\xi_{2}):=\sum_{n=0}^{\infty}\frac{1}{2^{n}}\frac{d^{(n)}(\xi_{1},\xi_{2})}{1+d^{(n)}(\xi_{1},\xi_{2})},\qquad\xi_{1},\xi_{2}\in\Xi,
\end{equation}
where
$$
d^{(n)}(\xi_{1},\xi_{2}):=\max_{s\in[-n,n]} \|\xi_1(s) - \xi_2(s)\|,\qquad n\in\mathbb{N}. 
$$

Let $\mathcal{H}_\Xi(g)$ be the hull of the non-autonomous function $g\in\Xi$. It is not hard to see that the shift semigroup $\{\theta_t: t\in \R\}$ satisfies \ref{lipdriving}.

We consider the operator $A = - \Delta: H^2(\O) \cap H^1_0(\O) \rightarrow L^2(\O)$, and we know that $-A$ is a positive and self-adjoint operator with compact resolvent. Moreover, $-A$ generates an analytic semigroup on $X$ that satisfies an exponential estimate:
\begin{equation}
\|e^{-At}\| \leq M e^{-\alpha t} 
\end{equation}
where $\alpha>0$. We also denote by $A^\theta$ the fractional powers of $A$, $\theta \in \R$, and by $X_\theta$ the Banach spaces $(D(A^\theta),\|A^\theta \cdot \| )$.

Let $f \in \mathcal{C}(\R, \R)$ be a locally Lipschitz continuous function satisfying the following growth and sign conditions:
\begin{itemize}
    \item Growth: $|f(u) - f(v)| \leq C|u-v| \left (|u|^{\rho - 1} + |v|^{\rho-1} + 1 \right )$ with $\rho \leq 1 +\frac{4}{N}$.
    \item Sign: For some $C_0 \in \R$, $C_1 \geq 0$, we have $uf(u) \leq - C_0 u^2 +C_1 |u|$ for all $u\in \R$.
\end{itemize}

Moreover, let $g\in \mathcal{C}(\R, X)$ be a bounded, uniformly continuous function, so that by Arzelà-Ascoli Theorem the set $\Sigma = \mathcal{H}_\Xi(g)$ is compact in $\Xi$. The results in \cite{linfinito} can be applied with the same proofs for the non-autonomous case, to show that \eqref{intazim} has a unique solution $v:[\tau, \infty) \to X$ for every initial values $\tau \in \R$, $v_\tau \in X$, and we have a system of processes $\{U_\sigma(t,\tau): t\geq \tau\}_{\sigma\in \Sigma}$, where $U_\sigma(t,\tau)v_\tau$ is the solution of the following problem:
\begin{equation} \label{marjax}
\begin{split}
&v_t  = \Delta v + f( v )+\sigma(t),\quad x\in\mathcal{O},\ t>\tau,\\
&v (x,\tau)= v _{\tau}(x),\quad x\in\mathcal{O},\\
&v (x,t)=0,\quad x\in\partial\mathcal{O},\ t\geqslant 0,
\end{split} 
\end{equation}

Moreover, this system of processes has an uniform attractor $\A_\Sigma \subset L^\infty(\O)$ that is bounded in $L^\infty(\O)$ and in $Y$. Then there exists $M \geq 0$ such that: 
$$
\sup_{u\in \mathcal{A}_\Sigma} \sup_{x\in \O} |u(x)| \leq M.
$$
Therefore, we can modify the function $f \in \mathcal{C}(\R, \R)$ in the set $\R \setminus (-M,M)$, in such a way that the new function $f_{\text{new}} \in \mathcal{C}^1(\R,\R)$ is globally Lipschitz and the problem 
\begin{equation}
\begin{split}
&v_t  = \Delta v + f_{\text{new}}( v )+g(t),\quad x\in\mathcal{O},\ t>\tau,\\
&v (x,\tau)= v _{\tau}(x),\quad x\in\mathcal{O},\\
&v (x,t)=0,\quad x\in\partial\mathcal{O},\ t\geqslant 0.
\end{split} 
\end{equation}
%
%\begin{equation} 
%\begin{array}{l}{{\displaystyle\frac{\partial v}{\partial t}  = \Delta v + f_{\text{new}}( v )+g(t),}}\\ {{ v (x,t)|_{t=\tau}= v _{\tau}(x),\quad v (x,t)|_{x\in \partial\mathcal{O}}=0,\qquad x\in\mathcal{O},\,t\geqslant\tau,}}\end{array} 
%\end{equation}
has the same uniform attractor $\A_\Sigma$ as the problem \eqref{intazim}. Moreover, the process $\{U_g(t,s):t\geq s\}$ has a pullback attractor $\{A(t):t\in \R\}$ such that $\A_p := \bigcup_{t\in R} A(t) \subset \A_\Sigma$ and $\A_p$ is compact in $L^2(\O)$.

This reasoning implies that we can study the uniform attractor of problem \eqref{intazim} without loss of generality with a globally Lipschitz function $f\in \mathcal{C}^1(\R, \R)$. From now on, suppose the nonlinearity $f$ is indeed globally Lipschitz, and that $g(\R)$ is precompact in $X$ — this implies, in particular, that $\{\sigma(t):t\in \R, \, \sigma \in \Sigma\}$ is bounded in $X$. We denote by $\{U_\sigma(t,\tau): t\geq \tau\}_{\sigma\in \Sigma}$ the system of processes generated by \eqref{intazim}, and for simplicity we denote $F_\sigma(t,v) = f(v)+ \sigma(t)$.

We suppose that there exist uniformly continuous and bounded functions $g_\pm \in \Xi$ such that:
\begin{equation}\label{guarax}
\|g(t)-g_-(t)\| \leq Q_1 e^{\eta_1 t}, \quad \forall t\in \R,
\end{equation}
\begin{equation}\label{mutambo}
\|g(t)-g_+(t)\| \leq Q_2 e^{-\eta_2 t}, \quad \forall t\in \R.
\end{equation}

Let $\Sigma_- = \mathcal{H}_\Xi(g_-)$, $\Sigma_+ = \mathcal{H}_\Xi(g_+)$. Then the reaction-diffusion equation \eqref{marjax} is well-posed in $X$ and its solutions generate systems $\{U_{\sigma}(t,\tau)\}_{\sigma\in\Sigma_-}$ and $\{U_{\sigma}(t,\tau)\}_{\sigma\in\Sigma_+}$ in $X$. If $\Lambda = \Sigma \cup \Sigma_+ \cup \Sigma_-$, we have $\{\sigma(t):t\in \R, \sigma \in \Lambda\}$ bounded in $X$. 

In addition, the systems $\{U_{\sigma}(t,\tau)\}_{\sigma\in\Sigma_-}$ and $\{U_{\sigma}(t,\tau)\}_{\sigma\in\Sigma_+}$ have uniform attractors $\mathcal{A}_{\Sigma_-}$ and $\mathcal{A}_{\Sigma_+}$, respectively. All the systems $\{U_\sigma(t,s):t\geq s\}_{\sigma \in \Sigma}$, $\{U_\sigma(t,s):t\geq s\}_{\sigma \in \Sigma_-}$ and $\{U_\sigma(t,s):t\geq s\}_{\sigma \in \Sigma_+}$ have uniformly absorbing sets bounded in $L^\infty(\mathcal{O})$ and in $Y$ (again, following the results on \cite{linfinito}). We call these sets $\mathcal{B}$, $\mathcal{B}_1$ and $\mathcal{B}_2$, respectively.   

We suppose that $\Sigma_-$ and $\Sigma_+$ have finite box-counting dimension. This can be shown when $g_-$ and $g_+$ are quasiperiodic functions, as it is presented in \cite{cuietal}.

Now we prove some continuity properties. First, it is easy to see using the variation of constants formula and Gronwall's inequality that:
\begin{equation}
\|U_\sigma(t,s)x - U_\sigma(t,s)y\| \leq M e^{\eta(t-s)} \|x-y\|,\quad t\geq s, \quad x, y\in X
\end{equation}
with $M$, $\eta \in \R$ independent of $\sigma \in \Lambda$.

Then, let us show the Hölder continuity in time. First, notice that if $K$ is bounded in $L^2(\O)$, the set:
\begin{equation*}
J_K = \{U_\sigma(t,0)x: t\geq 0, \, \sigma\in \Lambda, \, x\in K\}
\end{equation*}
is bounded in $L^2(\O)$ as well. Here, we used the existence of the uniform attractors $\A_\Sigma$, $\A_{\Sigma_-}$ and $\A_{\Sigma_+}$.

Let $0\leq \Tilde{t} \leq t \leq 1$, $x\in \B$, then, by formula of variation of constants, if $\delta \in (0,1/2)$:
\begin{equation*}
\begin{split}
\|U_\sigma(t+s,s)x-U_\sigma(\Tilde{t}+s,s)x\|  \leq & \left \|(e^{-At}-e^{-A\Tilde{t}})x \right \| + \left \| \int_{\Tilde{t}+s}^{t+s}e^{-A(t+s-\theta)}F_\sigma(\theta, U_\sigma(\theta,s)x)d\theta \right \| \\
& + \left \| \int_{s}^{\Tilde{t}+s} \left ( e^{-A(t+s-\theta)} - e^{-A(\Tilde{t}+s-\theta)} \right ) F_\sigma(\theta, U_\sigma(\theta,s)x) d\theta \right \|.
\end{split}
\end{equation*}

We will estimate the right hand side terms in the above equations one by one. For the first term, we have:
\begin{equation*}
\left \|(e^{-At}-e^{-A\Tilde{t}})x \right \| = \left \|(e^{-A(t-\Tilde{t})}-I)A^{-\delta}e^{-A\Tilde{t}}A^{\delta}x \right \| \leq C |t-\Tilde{t}|^{\delta} \sup_{x\in \B} \|A^\delta x\|  
\end{equation*}

For the second term, we have:
\begin{equation*}
\left \| \int_{\Tilde{t}+s}^{t+s}e^{-A(t+s-\theta)}F(\theta, U_\sigma(\theta,s)x)d\theta \right \| \leq C |t-\Tilde{t}| \sup_{x\in J_{\B}} \, \sup_{\theta\in \R} \, \sup_{\sigma\in \Lambda}\|F_\sigma(\theta,x)\|
\end{equation*}

And for the third term, we get:
\begin{equation*}
\begin{split}
& \left \| \int_{s}^{\Tilde{t}+s} \left ( e^{-A(t-\Tilde{t})} - I \right ) A^{-\delta} A^\delta e^{-A(\Tilde{t}+s-\theta)} F_\sigma(\theta, U_\sigma(\theta,s)x) d\theta \right \| \\
& \leq \left \| \int_{s}^{\Tilde{t}+s} C |t-\Tilde{t}|^\delta \frac{e^{-\alpha(\Tilde{t}+s-\theta)}}{(\Tilde{t}+s - \theta)^\delta} d\theta \,  \sup_{x\in J_{\B}} \, \sup_{\theta\in \R} \, \sup_{\sigma\in \Lambda}\|F_\sigma(\theta,x)\| \right \| \\
& \leq \Tilde{C} |t- \Tilde{t}|^\delta \sup_{x\in J_{\B}} \, \sup_{\theta\in \R} \, \sup_{\sigma\in \Lambda}\|F_\sigma(\theta,x)\|
\end{split}
\end{equation*}

Since $\{\sigma(t):t\in \R, \, \sigma \in \Lambda\}$ and $J_{\B}$ are bounded in $X$ and $f:X\to X$ is globally Lipschitz, it is easy to see that
$$
\sup_{x\in J_{\B}} \, \sup_{\theta\in \R} \, \sup_{\sigma\in \Lambda}\|F_\sigma(\theta,x)\|
$$
is finite. Then, we conclude that:
\begin{equation*}
\begin{split}
\|U_\sigma(t+s,s)x-U_\sigma(\Tilde{t}+s,s)x\|  \leq K |t-\Tilde{t}|^\delta
\end{split}
\end{equation*}
uniformly in $x\in \B$ and $\sigma \in \Lambda$, and $t$, $\tilde{t}$ in bounded sets of $\R$.

Using the variation of constants formula, it is not hard to show the smoothing property: if $D \subset X$ is bounded, there exists a function $\kappa: \R^+ \to \R^+$ such that:
$$
    \underset{\sigma\in \Lambda}{\sup} \|U_\sigma(t,0)x - U_\sigma(t,0)y\|_Y \leq \kappa(t) \|x-y\|_X, \quad \forall \, x,y \in D.
$$

Finally, with the aid of Lemma 5.1 in \cite{cuietal}, we can show that \ref{lipschitz} follows for every $\sigma_1, \sigma_2 \in \Lambda$.

\textit{Box-counting dimension of $\mathcal{A}_\Sigma$.}

Now we will estimate the box-counting dimension of the uniform attractor of \eqref{intazim} using Theorem \ref{inthuuuuglou}. 

Let $\sigma \in \omega(g)$, then there exists a sequence $t_n \to \infty$ such that $\theta_{t_n} g \to \sigma$. Then 
$$\|\theta_{t_n} g_+ - \sigma \|_\Xi \leq \|\theta_{t_n} g_+ - \theta_{t_n} g \|_\Xi + \|\theta_{t_n} g - \sigma \|_\Xi \overset{n\to \infty}{\longrightarrow} 0$$
where we used \eqref{mutambo} and the definition of $\Xi$.

It follows that $\theta_{t_n} g_+ \to \sigma$ and $\sigma\in \mathcal{H}_\Xi(g^+)$. We have just shown that $\omega(g) \subset \mathcal{H}_\Xi(g^+)$, and analogously, $\alpha(g) \subset \mathcal{H}_\Xi(g^-)$.

For the exponential estimates \eqref{intblux} and \eqref{intravael}, first remember that $\{\sigma(t):t\in \R, \sigma \in \Lambda\}$ is bounded in $X$, therefore we can increase $\eta_1$ and reduce $\eta_2$ in \eqref{guarax} and \eqref{mutambo} to satisfy $1<e^{\eta_1} < 2$ and $1 < e^{\eta_2} < 2$, from which follows that:
$$
\dist_\Xi(\theta(t) g, \theta(t) g_+) \leq \Bar{c}e^{-\eta_2 t}, \quad t\in \R
$$
$$
\dist_\Xi(\theta(t) g, \theta(t) g_-) \leq \Bar{c}e^{\eta_1 t}, \quad t\in \R
$$
where $\Bar{c}>0$ is a constant.

Therefore, all hypotheses of Theorem \ref{inthuuuuglou} apply and the uniform attractor $\mathcal{A}_\Sigma$ has finite box-counting dimension.

\subsection{An infinite dimensional symbol space with finite dimensional uniform attractor}

We will show that in \eqref{intazim}, the function $g$ can satisfy all the hypotheses we asked in Section \ref{app} and have an infinite-dimensional hull. This shows that our results can be used to bound the box-counting dimension of the uniform attractor associated to dynamical systems that have infinite dimensional symbol space, which is an advance in relation to the existing methods. 

First, take any compact infinite dimensional set $A \subset X$. For example, $A$ can be the closed unit ball of $H^1_0(\O)$ regarded as a subset of $X = L^2(\O)$. Then we use the classic topology result in \cite{immer} to guarantee the existence of a function $\alpha:\mathcal{C} \to A$ continuous and surjective, where $\mathcal{C}$ is the Cantor set:
$$
\mathcal{C} = \left\{ x \in [0,1] : x \text{ has a ternary expansion without using the digit } 1 \right\}
$$

Then, we can extend the function $\alpha$ to a function $\Tilde{\alpha}:[0,1] \to X$ defined on the interval $[0,1]$ using the following strategy: we use the value of $\alpha$ for points in the Cantor set $\mathcal{C}$, and in the removed intervals $(a,b) \subset [0,1] \setminus \mathcal{C}$, $a$, $b \in \mathcal{C}$, we define $\tilde{\alpha}$ as a linear function such that $\tilde{\alpha}(a) =\alpha(a)$ and $\tilde{\alpha}(b) =\alpha(b)$.

Then, we can extend the function $\tilde{\alpha}$ trivially in the line creating the function $g:\R \to X$ such that 
\begin{equation}\label{jarx}
g(t) = \begin{cases} 
0 & \text{for } t < -1 \text{ or } t > 2 \\
\tilde{\alpha}(t) & \text{for } t \in [0,1] \\
(1+t)\tilde{\alpha}(0) & \text{for } t \in [-1,0] \\
(2-t)\tilde{\alpha}(1) & \text{for } t \in [1,2]
\end{cases}
\end{equation}

It is obvious, then, that $g(\R)$ contains $A$ and therefore has infinite box-counting dimension.

Consider now the map $\Psi: \{\theta_t g : t\in \R\} \to X$ given by $\Psi(\theta_t g) = g(t)$, $t\in \R$. It is easy to see that this map is well-defined and Lipschitz continuous from $(\Xi, d_\Xi)$ into $X$, where $d_\Xi$ is the Fréchet metric defined on \eqref{albatroz}. Therefore, 
$$
d_B(g(\R)) \leq d_B(\{\theta_t g : t\in \R\},
$$

Then $\{\theta_t g : t\in \R\}$, and $\Sigma 
= \mathcal{H}_\Xi(g)$ must have infinite box-counting dimension. Nevertheless, it is easy to see that we can apply the reasoning used in Section \ref{app} with $g_+$ and $g_-$ equal to the zero function $\mathbf{0}:\R \to X$. The asymptotic hulls will be trivial: $\Sigma_-=\Sigma_+ = \mathcal{H}_\Xi(\mathbf{0}) = \{\mathbf{0}\}$, which clearly are finite dimensional.

Therefore, if we consider equation \eqref{intazim} with the corresponding hypotheses and the specific forcing function $g$ given by \eqref{jarx}, the differential equation will generate a system of processes $\{U_\sigma(t,s):t\geq s\}_{\sigma\in \mathcal{H}_\Xi(g)}$ that possesses a finite box-counting dimension uniform attractor $\mathcal{A}_\Sigma$ — as showed in Section \ref{app}, even though the symbol space $\Sigma = \mathcal{H}_\Xi(g)$ has infinite box-counting dimension in $\Xi$.

	\bibliographystyle{acm}

	%\nocite{*}
	\bibliography{bibliography}
	
\end{document}